\documentclass[12pt,reqno]{amsart}

\usepackage{XCharter}

\usepackage{amsmath, amssymb, amsthm}
\usepackage{mathtools}        
\usepackage{mathrsfs}         
\usepackage{stmaryrd}         

\usepackage{tikz}
\usetikzlibrary{cd, decorations.pathmorphing, arrows, positioning}

\usepackage[top=1in, bottom=1in, left=2.5cm, right=2.5cm]{geometry}
\usepackage{parskip}

\usepackage{microtype}

\usepackage[shortlabels]{enumitem}  

\usepackage{appendix}
\usepackage{verbatim}
\usepackage{comment}
\usepackage{graphicx}
\usepackage{subfiles}
\usepackage{chngcntr}

\usepackage[colorlinks, pagebackref]{hyperref}
\hypersetup{
  colorlinks = true,
  citecolor  = red,
  linkcolor  = blue,
  urlcolor   = cyan
}

\theoremstyle{definition}
\newtheorem{para}{}[section]

\theoremstyle{plain}
\newtheorem{thm}[para]{Theorem}
\newtheorem{lem}[para]{Lemma}
\newtheorem{cor}[para]{Corollary}

\theoremstyle{definition}

\theoremstyle{remark}
\newtheorem{rem}[para]{Remark}

\newcommand{\bbA}{\mathbb{A}}  
\newcommand{\bbC}{\mathbb{C}}  
  \newcommand{\bbF}{\mathbb{F}}
\newcommand{\bbG}{\mathbb{G}}

  \newcommand{\bbP}{\mathbb{P}}
\newcommand{\bbQ}{\mathbb{Q}}

  \newcommand{\cH}{\mathcal{H}}
  
  \newcommand{\cL}{\mathcal{L}}
\newcommand{\cM}{\mathcal{M}}

\newcommand{\HH}{\mathrm{H}}

\newcommand{\ol}[1]{\overline{#1}}

\newcommand{\Ql}{\bbQ_{\ell}}

\DeclareMathOperator{\Tor}{Tor}

\newcommand{\Db}{D^b_c}
\newcommand{\Gm}{\bbG_m}
\newcommand{\Perv}{\operatorname{Perv}}
\newcommand{\Supp}{\operatorname{Supp}}
\newcommand{\cu}{c_u}
\newcommand{\Cone}{\operatorname{Cone}}
\newcommand{\Eig}{\operatorname{Eig}}

\newif\ifasens
\ifdefined\ASENS\asenstrue\else\asensfalse\fi

\title{Uniform stratified vanishing and equidistribution on \(\Gm^d\)}

\author{Amadou Bah}
\address{Department of Mathematics, University of Notre Dame, Notre Dame, IN 46556, USA}
\email{abah2@nd.edu}

\author{K.~V. Shuddhodan}
\address{Department of Mathematics, University of Notre Dame, Notre Dame, IN 46556, USA}
\email{skadattu@nd.edu}


\begin{document}

\begin{abstract}
We prove a uniform stratified generic vanishing theorem for perverse sheaves on \(\Gm^d\) over
finite fields, with constants depending only on the dimension and on the complexity \`a la
Sawin.  As a corollary, we prove an equidistribution theorem for \(\Gm^d\)
extending Katz's theorem for \(\Gm\).

The corollary applies to sequences of perverse sheaves of
bounded complexity with common tannakian monodromy group, over finite fields of
\emph{varying} characteristic.  As the cardinalities of the fields tend to infinity, the Frobenius
conjugacy classes attached to the multiplicative characters equidistribute in the space of
conjugacy classes of a maximal compact subgroup.
\end{abstract}

\maketitle

\ifasens
\begin{center}
\begin{minipage}{0.88\textwidth}
\small
\noindent\textsc{R\'esum\'e}.  Nous d\'emontrons un th\'eor\`eme d'annulation g\'en\'erique stratifi\'ee
uniforme pour les faisceaux pervers sur \(\Gm^d\) au-dessus d'un corps fini, les constantes ne d\'ependant
que de la dimension et de la complexit\'e \`a la Sawin.  Nous en d\'eduisons un th\'eor\`eme
d'\'equir\'epartition sur \(\Gm^d\) qui \'etend le th\'eor\`eme de Katz pour \(\Gm\).

\smallskip
\noindent Ce corollaire s'applique aux suites de faisceaux pervers de complexit\'e born\'ee ayant un
m\^eme groupe de monodromie tannakien, sur des corps finis de caract\'eristique \emph{variable}.
Lorsque les cardinaux des corps tendent vers l'infini, les classes de conjugaison de Frobenius
attach\'ees aux caract\`eres multiplicatifs s'\'equir\'epartissent dans l'espace des classes de
conjugaison d'un sous-groupe compact maximal.
\end{minipage}
\end{center}
\medskip
\fi
\setcounter{tocdepth}{1}
\tableofcontents

\section{Introduction}
\label{sec:statement}

Let \(k=\bbF_q\) be a finite field of characteristic \(p\), let \(\bar k\) be a fixed algebraic closure of
\(k\), let \(\ell\neq p\) be a prime, and let \(\Gm^d/k\) be the split \(d\)-dimensional torus.  We write
\(\Gm^d/\bar k\) for its geometric base change and use unadorned \(\Gm^d\) for this geometric torus.
For \(n\ge1\), \(k_n=\bbF_{q^n}\) denotes the extension of \(k\) of degree \(n\) inside \(\bar k\).

Through the Lang isogeny, a character
\(\chi:\Gm^d(k_n)\to\ol\Ql^\times\) determines a rank-one Kummer local system \(\cL_\chi\) on
\(\Gm^d\).  If \(M_0\) is perverse on \(\Gm^d/k\) and \(M\) is its geometric base change, we set
\(M_\chi=M\otimes\cL_\chi\).  To this we can associate a character sum
\[
        S(M_0;k_n,\chi)=\sum_{x\in \Gm^d(k_n)}t_{M_0}(x;k_n)\,\chi(x),
\]
where \(t_{M_0}\) is the trace function of \(M_0\).  This class includes Gauss sums, Kloosterman sums,
and their multivariable analogues.  By the Grothendieck--Lefschetz trace formula, the sum is the
alternating Frobenius trace on \(R\Gamma_c(\Gm^d,M_\chi)\).  Generic vanishing theorems describe how this
cohomology varies with \(\chi\).  Tannakian methods turn this description into equidistribution theorems
for the sums.

\begin{enumerate}[(a)]

\item Gabber--Loeser proved generic vanishing on tori \cite[Cor.~2.3.2]{GabberLoeser1996}.  Assuming
resolution of singularities and simplification of ideals for varieties of dimension \(<d\), they proved
that the locus where the forget-supports morphism
\(R\Gamma_c(\Gm^d,M_\chi)\to R\Gamma(\Gm^d,M_\chi)\) is \emph{not} an isomorphism is contained in a
finite union of translated algebraic cotori of the character scheme
\(\widehat{\Gm^d}\) \cite[Th.~4.1.1']{GabberLoeser1996}.\footnote{The
\(\ol\Ql\)-scheme of \(\ol\Ql^{\times}\)-characters of \(\Gm^d\), whose \(\ol\Ql\)-points parametrize
rank-one Kummer local systems.  The construction is recalled in \S\ref{sec:characters}.}
Forey--Fres\'an--Kowalski removed these hypotheses \cite[Th.~2.12]{ForeyFresanKowalski2021}.

\item For \(d=1\), Katz developed a tannakian theory on \(\Gm\) and proved equidistribution theorems for
these character sums \cite{Katz2012}.

\item Forey, Fres\'an, and Kowalski \cite{ForeyFresanKowalski2021} extended the tannakian framework to
every connected commutative algebraic group over a finite field.  For a fixed generically unramified,
arithmetically semisimple\footnote{A perverse sheaf on \(\Gm^d/k\) is \emph{arithmetically semisimple}
if it is semisimple as a perverse sheaf over \(k\) \cite[\S1.12]{ForeyFresanKowalski2021}.  On a torus
every arithmetically semisimple object is generically unramified
\cite[Th.~3.27]{ForeyFresanKowalski2021}, and Zurbuchen \cite{Zurbuchen2026} has recently proved generic
unramifiedness for every perverse sheaf on a connected commutative algebraic group.} object \(M_0\) pure
of weight zero,
they proved that the classes
\(\Theta_{M_0,k_n}(\chi)\) attached to the unramified characters equidistribute on average as
\(n\to\infty\) \cite[Th.~4.11]{ForeyFresanKowalski2021}.

\end{enumerate}

In the equidistribution results above, the object is fixed and the field varies through its finite
extensions.  This article concerns equidistribution along arbitrary sequences of finite fields of
growing cardinality, allowing the characteristic to vary.

Katz proved such a theorem in dimension one, for forms of \(\Gm\)
\cite[Th.~28.1]{Katz2012} (see also the Bourbaki seminar of Fres\'an \cite{FresanBourbaki2019}).  For
arbitrary powers of the additive group, an equidistribution theorem over prime fields follows from
the result for lisse sheaves of \cite[Th.~7.22]{SFFK23}, see
\cite[Rem.~4.20(1)]{ForeyFresanKowalski2021}.  Forey--Fres\'an--Kowalski developed the
corresponding theory for multiplicative characters of connected commutative algebraic groups, including
their horizontal equidistribution theorem \cite[Th.~4.19]{ForeyFresanKowalski2021}.  For tori of
dimension at least two, this theorem was left open
\cite[Rem.~4.20(2)]{ForeyFresanKowalski2021}.  In the present article we prove it for \(\Gm^d\) in every
dimension.

\subsection{Statements}
\label{sec:statements}

\subsubsection{Uniform stratified generic vanishing}

We fix notation.  For \(K\in\Db(\Gm^d,\ol\Ql)\), \(\cu(K)\) denotes Sawin's complexity
relative to the embedding \(u_d\) fixed in \S\ref{sec:embeddings} and recalled in
\S\ref{sec:complexity}.  We denote by \(\widehat{\Gm^d}\) the Gabber--Loeser scheme of
\(\ol\Ql^{\times}\)-characters of \(\Gm^d\) (\S\ref{sec:characters}).  Its
\(\ol\Ql\)-points parametrize the Kummer sheaves \(\cL_\chi\), and we put \(K_\chi=K\otimes\cL_\chi\).
Applied to the structure morphism of \(\Gm^d\) and the complex \(K_\chi\), the transformation
\eqref{eq:forget-supports} gives the morphism
\[
        R\Gamma_c(\Gm^d,K_\chi)
        \xrightarrow{\varphi_{K,\chi}}
        R\Gamma(\Gm^d,K_\chi).
\]
Following Gabber--Loeser, let \(\Delta(K)\subseteq\widehat{\Gm^d}\) be the support of the cone of the
comparison morphism between the two Mellin transforms of \(K\), as recalled in \ref{item:GL2} of
\S\ref{sec:mellin}.  By \cite[Prop.~3.3.4(i)]{GabberLoeser1996},
\[
        \Delta(K)(\ol\Ql)
        =\{\chi\in\widehat{\Gm^d}(\ol\Ql):\varphi_{K,\chi}\text{ is \emph{not} an isomorphism}\}.
\]
For \(n\ge1\) we write
\(\widehat{\Gm^d(k_n)}=\operatorname{Hom}\bigl(\Gm^d(k_n),\ol\Ql^{\times}\bigr)\), and for a subset
\(Z\subseteq\widehat{\Gm^d}(\ol\Ql)\) we put \(Z(k_n)=Z\cap\widehat{\Gm^d(k_n)}\).

If \(A\in\Perv(\Gm^d,\ol\Ql)\), let \(X_w(A)\subseteq\widehat{\Gm^d}(\ol\Ql)\) be the weakly unramified
locus, defined by the conditions of \cite[Def.~3.1]{ForeyFresanKowalski2021}.\footnote{For the perverse sheaf \(A_\chi\) on the
affine torus, Artin vanishing and its Verdier-dual form \cite[Th.~4.1.1, Cor.~4.1.2]{BBDG18} show that
\(\varphi_{A,\chi}\) is an isomorphism exactly when the conditions of loc.~cit. hold.  Thus
\(\Delta(A)(\ol\Ql)\) is the complement of the weakly unramified locus.}
\begin{equation}
\label{eq:weak-locus}
 X_w(A)=\widehat{\Gm^d}(\ol\Ql)\setminus\Delta(A)(\ol\Ql).
\end{equation}

\begin{thm}[uniform point count for \(\Delta(K)\)]
\label{thm:delta-count}
For every \(d\ge1\) there is a nondecreasing function
\(\Phi^\Delta_{d}:\mathbb R_{\ge0}\to\mathbb R_{\ge0}\) with the following property.  For every
\(C\ge0\), every finite field \(k=\bbF_q\), every prime \(\ell\neq\operatorname{char}k\), every
\(K\in\Db(\Gm^d,\ol\Ql)\) with \(\cu(K)\le C\), and every \(n\ge1\),
\[
        |\Delta(K)(k_n)|\le \Phi^\Delta_{d}(C)\,q^{n(d-1)} .
\]
\end{thm}

For \(A\in\Perv(\Gm^d,\ol\Ql)\) and \(0\le i\le d\), put
\begin{equation}
\label{eq:cohomology-jump-loci}
\begin{aligned}
 \mathcal V_{d,i}(A)=\bigl\{\chi\in\widehat{\Gm^d}(\ol\Ql):{}&
 \HH^{i}_c(\Gm^d,A_\chi)\neq0\ \text{or}\ \HH^{-i}_c(\Gm^d,A_\chi)\neq0\ \text{or}\\
 &\HH^{i}(\Gm^d,A_\chi)\neq0\ \text{or}\ \HH^{-i}(\Gm^d,A_\chi)\neq0\bigr\}.
\end{aligned}
\end{equation}
For \(i>d\), we use the convention \(\mathcal V_{d,i}(A)=\varnothing\).

\begin{thm}[bounded-complexity stratified point count]
\label{thm:stratified}
For every \(d\ge1\) there is a nondecreasing function
\(\Psi_{d}:\mathbb R_{\ge0}\to\mathbb R_{\ge0}\) such that for every \(C\ge0\), every finite field
\(k=\bbF_q\), every prime \(\ell\neq\operatorname{char}k\), every \(A\in\Perv(\Gm^d,\ol\Ql)\) with
\(\cu(A)\le C\), every \(0\le i\le d\), and every \(n\ge1\),
\[
        |\mathcal V_{d,i}(A)(k_n)|\le\Psi_{d}(C)\,q^{n(d-i)} .
\]
\end{thm}

Forey--Fres\'an--Kowalski propose this bounded-complexity refinement of their stratified
vanishing theorem \cite[Th.~2.3, Rem.~2.4]{ForeyFresanKowalski2021}.\footnote{For tori, the
non-uniform stratified vanishing theorem is
\cite[Th.~3.4.1, Prop.~6.3.2, Th.~6.3.3]{GabberLoeser1996}, again conditional on the hypotheses
recalled above, removed in \cite[Cor.~2.16]{ForeyFresanKowalski2021}.}  Its case \(i=1\) bounds
\(\mathcal V_{d,1}(A)\), a subset of
\(\Delta(A)(\ol\Ql)\), and is subsumed by Theorem~\ref{thm:delta-count}, which bounds the larger set
\(\Delta(A)(\ol\Ql)\).

\subsubsection{Equidistribution}
\label{sec:equidistribution}

The remaining statements use the tannakian framework of \cite{ForeyFresanKowalski2021}.  The categories
\(\Perv_{\mathrm{int}}(\Gm^d)\), \(\Perv_{\mathrm{int}}^{\mathrm{ari}}(\Gm^d/k)\), and the tannakian
subcategory \(\langle M\rangle\) of \((\Perv_{\mathrm{int}}(\Gm^d),\ast_{\mathrm{int}})\) generated by an
object \(M\) are recalled in \S\S\ref{sec:negligible}--\ref{sec:tannakian}.

For \(M_0\in\Perv_{\mathrm{int}}^{\mathrm{ari}}(\Gm^d/k)\), with geometric base change
\(M=M_{0,\bar k}\), let \(X(M_0)=X(M)\subseteq X_w(M)\) denote the unramified locus of
\cite[Def.~3.25, \S3.8]{ForeyFresanKowalski2021}.  Its definition in terms of fibre functors on
\(\langle M\rangle\) is recalled in \S\ref{sec:tannakian}.
Following \cite[Def.~1.22]{ForeyFresanKowalski2021}, a subset
\(X\subseteq\widehat{\Gm^d}(\ol\Ql)\) is \emph{generic} if there exists a constant \(c(X)\ge0\) such that
\[
        |(\widehat{\Gm^d}(\ol\Ql)\setminus X)(k_n)|\le c(X)q^{n(d-1)}
\]
for every \(n\ge1\).  The object \(M_0\) is \emph{generically unramified} if \(X(M_0)\) is generic.

\begin{thm}[unramified point count]
\label{thm:unram-count}
For every \(d\ge1\) there is a nondecreasing function
\(\Phi^{\mathrm{ur}}_{d}:\mathbb R_{\ge0}\to\mathbb R_{\ge0}\) with the following property.  For
every \(C\ge0\), every finite field \(k=\bbF_q\), every prime
\(\ell\neq\operatorname{char}k\), every arithmetically semisimple object
\(M_0\in\Perv_{\mathrm{int}}^{\mathrm{ari}}(\Gm^d/k)\) whose geometric base change
\(M=M_{0,\bar k}\) satisfies \(\cu(M)\le C\), and every \(n\ge1\),
\[
        |(\widehat{\Gm^d}(\ol\Ql)\setminus X(M_0))(k_n)|
        \le \Phi^{\mathrm{ur}}_{d}(C)\,q^{n(d-1)} .
\]
\end{thm}

Fix a prime \(\ell\) and an isomorphism \(\iota\colon\ol\Ql\to\bbC\), through which the unitary
Frobenius conjugacy classes below are formed, as in \cite[\S3.9]{ForeyFresanKowalski2021}.  Their
construction is recalled in \ref{item:T5} of \S\ref{sec:tannakian}.  For an integer
\(r\ge1\) and a reductive subgroup \(G\subseteq\operatorname{GL}_{r,\ol\Ql}\), let \(K_G\) be a maximal
compact subgroup of \(G(\bbC)\), let \(K_G^\sharp\) be its space of conjugacy classes, and let
\(\mu_G^\sharp\) be the pushforward of the normalized Haar probability measure on \(K_G\).

Let \((k_j)_{j\ge1}\) be a sequence of finite fields of characteristic \(\neq\ell\) with \(\#k_j\to\infty\).
For each \(j\), choose an algebraic closure \(\bar k_j\) of \(k_j\), and let
\(M_{j,0}\in\Perv_{\mathrm{int}}^{\mathrm{ari}}(\Gm^d/k_j)\) be arithmetically semisimple and pure of
weight zero, with geometric base change \(M_j=M_{j,0,\bar k_j}\).  Let \(r_j\) be the tannakian dimension
of \(M_j\), and put
\[
 G_j^{\mathrm{geo}}=G^{\mathrm{geo}}_{M_j},\qquad
 G_j^{\mathrm{ari}}=G^{\mathrm{ari}}_{M_j},\qquad
 X_j=X(M_{j,0})(k_j).
\]
Thus \(G_j^{\mathrm{geo}}\subseteq G_j^{\mathrm{ari}}\subseteq\operatorname{GL}_{r_j,\ol\Ql}\) are the
geometric and arithmetic tannakian groups of \(M_j\) \cite[Ch.~3]{ForeyFresanKowalski2021}.  They are
reductive by \ref{item:T2} in \S\ref{sec:tannakian}.

\begin{cor}
\label{cor:equidistribution}
Assume that \(\sup_j\cu(M_j)<\infty\), that \(r_j=r\) for some integer \(r\), and that
\(G_j^{\mathrm{ari}}=G_j^{\mathrm{geo}}\) for every \(j\).  Suppose that these groups are conjugate in
\(\operatorname{GL}_{r,\ol\Ql}\) to a fixed reductive subgroup \(G\).  After choosing such conjugacies,
the unitary Frobenius conjugacy classes \(\Theta_{M_{j,0},k_j}(\chi)\), \(\chi\in X_j\), become
\(\mu_G^\sharp\)-equidistributed in \(K_G^\sharp\) as \(j\to\infty\).
\end{cor}

For \(d\ge2\), the horizontal equidistribution theorem of Forey--Fres\'an--Kowalski was conditional on
uniform bounds for the stratified cohomology-jump loci and the complement of the unramified locus (see
\cite[Th.~4.19, Rem.~4.20(2)]{ForeyFresanKowalski2021}).  Theorems~\ref{thm:stratified}
and~\ref{thm:unram-count} provide these bounds for \(\Gm^d\), and
Corollary~\ref{cor:equidistribution} gives the resulting unconditional theorem.

For \(d=1\), Corollary~\ref{cor:equidistribution} recovers Katz's equidistribution theorem
\cite[Th.~28.1]{Katz2012} for \(\Gm\).\footnote{Katz assumes instead that the \(M_j\) have a common
tannakian dimension and uniformly bounded generic rank, and these hypotheses already bound
\(\cu(M_j)\) uniformly.  For the degree-one embedding \(u_1\) of \(\Gm\) in \(\bbP^1\), the
curve-complexity formula \cite[Th.~7.3(2)]{SFFK23} gives
\(\cu(M_j)\le4\operatorname{rank}(M_j)+\operatorname{loc}(M_j)\), and \(\chi_c(\Gm,\ol\Ql[1])=0\)
reduces the Grothendieck--Ogg--Shafarevich formula to
\(\operatorname{loc}(M_j)=\chi_c(\Gm,M_j)\), which is the tannakian dimension of \(M_j\).  We thank
Emmanuel Kowalski for this observation.}  Like Katz's theorem, it applies along arbitrary sequences of
finite fields, covering the extension towers, the sequences of prime fields, and every amalgam of the
two.

\begin{rem}
\label{rem:sum-equidistribution}
Corollary~\ref{cor:equidistribution} has a companion at the level of the sums, the analogue on
\(\Gm^d\) along arbitrary sequences of finite fields of \cite[Th.~4.8]{ForeyFresanKowalski2021}.  It is
one of the variants that Forey--Fres\'an--Kowalski leave to the reader before
\cite[Th.~4.19]{ForeyFresanKowalski2021}.  Let \(\nu\) be the direct image of \(\mu_G^\sharp\) under the
trace of the standard representation \(K_G^\sharp\to\bbC\).  For \(\chi\in X_j\), the
Grothendieck--Lefschetz trace formula and \ref{item:T5} give
\[
        S(M_{j,0};k_j,\chi)=\operatorname{Tr}\Theta_{M_{j,0},k_j}(\chi).
\]
Since the trace is continuous on \(K_G^\sharp\), the sums \(S(M_{j,0};k_j,\chi)\), \(\chi\in X_j\),
become \(\nu\)-equidistributed in \(\bbC\) as \(j\to\infty\).  By Theorem~\ref{thm:unram-count} the
characters outside \(X_j\) make up a vanishing fraction of all characters of \(\Gm^d(k_j)\), so the
same holds for the sums indexed by the full character group.
\end{rem}

Zurbuchen \cite{Zurbuchen2026} has recently proved generic unramifiedness for perverse sheaves on an
arbitrary connected commutative algebraic group, together with stratification theorems for the
associated exponential sums and relative weak propagation theorems.  His stratification theorems refine
those of Forey--Fres\'an--Kowalski in that the ramification loci are closed and the stratifications hold
uniformly in a family indexed by a scheme.  The estimates there are for a fixed object, with implied
constants depending on the complex and on the subset of the character space being counted
\cite[Prop.~2.84]{Zurbuchen2026}.  Theorems~\ref{thm:delta-count}, \ref{thm:stratified}
and~\ref{thm:unram-count} are of a different nature.  Their constants depend only on the complexity
\(\cu\), uniformly as the perverse sheaf, the characteristic, and \(\ell\) vary.

\subsection{Outline of the proofs}
\label{sec:outline}

The proof has three parts, a uniform Mellin estimate in one multiplicative variable, an induction on
the dimension of the torus, and a tannakian deduction of the equidistribution corollary from the
resulting point counts.  Figure~\ref{fig:leitfaden} shows the dependencies.

\begin{figure}[t]
\centering
\resizebox{\textwidth}{!}{%
\begin{tikzpicture}[
        >=stealth,
        every node/.style={align=center},
        main/.style={draw, rounded corners, thick, inner sep=5pt, font=\small},
        hub/.style={draw, rounded corners, thick, inner sep=5pt, font=\small},
        ext/.style={font=\footnotesize\itshape, text=black!65},
        dep/.style={->, thick},
        lab/.style={font=\footnotesize, fill=white, inner sep=1.5pt},
]
\node[ext] (kl)   at (-4.2, 5.2) {boundary stalks as\\inertia cohomology};
\node[ext] (cpx)  at ( 4.2, 5.2) {complexity\\\ref{item:C1}--\ref{item:C7}};
\node[hub] (exc) at (0,4.0)
        {Lemma~\ref{lem:exceptional} \textup{(\S\ref{sec:relative-mellin})} and\\
         Lemma~\ref{lem:complexity-bounds}\textup{(iii)}\\[2pt]
         $\#E_{\Gm}(A)\le F_{d}(\cu(A))$};
\node[main] (t1) at (-2.5,1.6) {Theorem~\ref{thm:delta-count}\\counting $\Delta(K)$};
\node[main] (t2) at ( 2.9,1.6) {Theorem~\ref{thm:stratified}\\stratified vanishing};
\node[ext]  (amp) at ( 6.8,1.6) {perverse amplitude\\Lemma~\ref{lem:relative-amplitude}};
\node[main] (t3)  at (-2.5,-0.8) {Theorem~\ref{thm:unram-count}\\unramified locus};
\node[main] (cor) at ( 0.2,-3.1) {Corollary~\ref{cor:equidistribution}\\equidistribution};
\draw[dep] (kl)   -- (exc);
\draw[dep] (cpx)  -- (exc);
\draw[dep] (exc) -- (t1) node[lab, pos=0.58] {induction on $d$ \textup{(\S\ref{sec:counting})}};
\draw[dep] (exc) -- (t2);
\draw[dep] (amp) -- (t2);
\draw[dep] (t1) -- (t3) node[lab, pos=0.5] {applied to $L=M\oplus M^{\vee}$};
\draw[dep] (t3) -- (cor);
\draw[dep] (t2) -- (cor) node[lab, pos=0.30] {Weyl criterion};
\end{tikzpicture}%
}
\caption{Leitfaden.  Dependencies among Theorems~\ref{thm:delta-count}--\ref{thm:unram-count},
Corollary~\ref{cor:equidistribution}, and their principal inputs.}
\label{fig:leitfaden}
\end{figure}
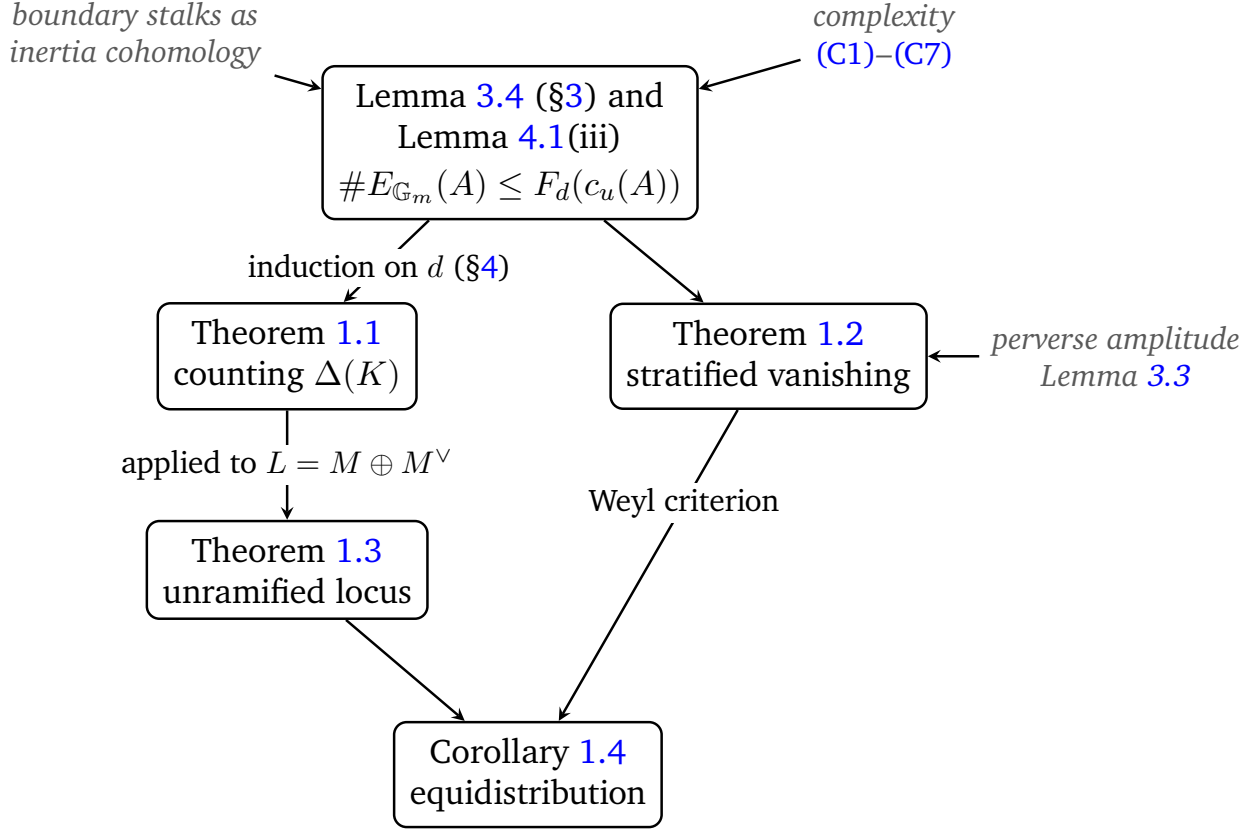

Gabber--Loeser prove generic vanishing on a torus by locating the \emph{bad} characters in a finite
union of translates of algebraic cotori, read off from the boundary monodromy along a normal crossing
compactification \cite[Th.~4.1.1']{GabberLoeser1996}.  Forey--Fres\'an--Kowalski prove the same
structure unconditionally, using de Jong's alterations
\cite[Th.~2.12, Prop.~2.13]{ForeyFresanKowalski2021}.  The equidistribution applications require information
of a different kind, namely bounds on the number of finite-field characters in these loci, uniform in
the characteristic, in \(\ell\), and in the sheaf, the sheaf entering only through its complexity.

A strategy discussed by Forey--Fres\'an--Kowalski in \cite[Rem.~2.4]{ForeyFresanKowalski2021} is to make the
cotorus description effective.  The obstruction is that de Jong's alterations give no control on the
number of exceptional components, so this would require an effective form of de Jong's theorem or
embedded resolution of singularities.  We obtain the required uniform bounds without alterations or
resolution of singularities.  The main new input is a uniform one-coordinate Mellin estimate proved
from boundary inertia on relative nearby cycles.  Applying this estimate successively along the
coordinate projections of \(\Gm^d\) gives the point counts in
Theorems~\ref{thm:delta-count}--\ref{thm:unram-count}.

We now describe the one-coordinate estimate.  Write
\(\Gm^d=\Gm\times\Gm^{d-1}\), with \(q\) and \(p\) the projections to \(\Gm\) and to \(\Gm^{d-1}\), so
that a character of \(\Gm^d\) is a pair \((\lambda,\eta)\) of characters of the two factors.  For a
perverse sheaf \(A\) on \(\Gm^d\), twisting by \(\lambda\) and pushing forward to \(\Gm^{d-1}\)
produces the two complexes
\[
        B_{!,\lambda}=Rp_!(A\otimes q^*\cL_\lambda),
        \qquad B_{*,\lambda}=Rp_*(A\otimes q^*\cL_\lambda).
\]
Call \(\lambda\) \emph{exceptional} for \(A\) if the forget-supports morphism
\(B_{!,\lambda}\to B_{*,\lambda}\) is not an isomorphism, and let \(E_{\Gm}(A)\subseteq\widehat{\Gm}\)
be the set of exceptional \(\lambda\).  Lemma~\ref{lem:exceptional}, through
Lemma~\ref{lem:complexity-bounds}(iii), gives the quantitative estimate
\[
        \#E_{\Gm}(A)\le F_{d}(\cu(A)),
\]
where \(F_{d}\) is the function fixed after Lemma~\ref{lem:complexity-bounds}.  For
\(\lambda\notin E_{\Gm}(A)\), the common value \(B_\lambda\) of \(B_{!,\lambda}\) and
\(B_{*,\lambda}\) is a perverse sheaf on \(\Gm^{d-1}\) with \(\cu(B_\lambda)\le F_{d}(\cu(A))\).

The estimate for \(E_{\Gm}(A)\) is a boundary calculation.  Compactify the first factor to
\(\bbP^1\).  The localization triangle reduces the failure of \(B_{!,\lambda}\to B_{*,\lambda}\) to
the boundary components \(\{0,\infty\}\times\Gm^{d-1}\).  The Hochschild--Serre spectral sequence
for nearby cycles \cite[\textit{Th\'eor\`emes de finitude}, Lem.~3.11]{SGA45} identifies the
corresponding stalks with continuous inertia cohomology of the relative nearby-cycle complexes.
By the prime-to-\(\ell\) vanishing of Appendix~\ref{app:boundary}, this cohomology is carried by
the pro-\(\ell\) quotient of tame inertia.
Twisting by \(\cL_\lambda\) multiplies the tame inertia action at \(x\in\{0,\infty\}\) by the scalar
\(\lambda_x(\tau_x)\), in the local monodromy notation of Lemma~\ref{lem:kummer-local-monodromy}.
Consequently, if \(\lambda\) is exceptional, then \(\lambda_x(\tau_x)^{-1}\) occurs as a monodromy
eigenvalue on a relative nearby-cycle complex at one of the two boundary points.  Since
\[
        \lambda\longmapsto\lambda_x(\tau_x)
\]
is injective, the exceptional characters are bounded by the number of such eigenvalues.  The
complexity estimates for nearby cycles and eigenvalue strata bound this number solely in terms of
\(\cu(A)\) and \(d\), uniformly in the field and in \(\ell\).  This is a relative and quantitative
version of the mechanism in \cite[(6.5.1)]{KatzLaumon1985}.  The finite exceptional set of the
qualitative argument is replaced by a bound depending only on the complexity.

Theorem~\ref{thm:stratified} and the perverse case of Theorem~\ref{thm:delta-count} are proved by
separate inductions on \(d\), both starting from Corollary~\ref{cor:d-one} and both using the estimate
above.  By the projection formula, twisting \(B_\lambda\) by \(\eta\)
computes the twisted cohomology of \(A\) at \((\lambda,\eta)\), with and without compact supports.
For non-exceptional \(\lambda\), the counting problem for \(A\) at the characters \((\lambda,\eta)\)
is therefore exactly the counting problem for \(B_\lambda\) in dimension \(d-1\)
(Lemma~\ref{lem:fibre}), and the inductive bound applies to each of the at most \(q^n\) such
\(\lambda\).  The exponent in Theorem~\ref{thm:stratified} is preserved by
\[
        q^n\cdot q^{n((d-1)-i)}=q^{n(d-i)}.
\]

For the perverse case of Theorem~\ref{thm:delta-count} the exceptional characters may be estimated
trivially in the remaining \(d-1\) variables, contributing at most
\(\#E_{\Gm}(A)\,|\widehat{\Gm^{d-1}(k_n)}|\le F_d(\cu(A))\,q^{n(d-1)}\).  For
Theorem~\ref{thm:stratified} in degree \(i>1\) this is too large, of order \(q^{n(d-1)}\) rather than
\(q^{n(d-i)}\), and the exceptional fibres must enter the induction as well.  Artin affine vanishing
and the relative-dimension bound place \(B_{!,\lambda}\) in perverse degrees \([0,1]\) and
\(B_{*,\lambda}\) in \([-1,0]\) (Lemma~\ref{lem:relative-amplitude}).  The perverse spectral sequence
therefore converts a cohomology jump in degree \(\pm i\) at \((\lambda,\eta)\) into a jump in degree
\(i\) or \(i-1\) at \(\eta\) for one of the at most four perverse cohomology sheaves of
\(B_{!,\lambda}\) and \(B_{*,\lambda}\), each of complexity at most \(F_d(\cu(A))\), and the
induction applies to these.  The term \(i-1\) gives \(q^{n((d-1)-(i-1))}=q^{n(d-i)}\), which closes
the induction.  The case of an arbitrary complex in Theorem~\ref{thm:delta-count} then follows from
the perverse cohomology sheaves.

Theorem~\ref{thm:unram-count} requires no new boundary estimate.  Unramifiedness follows from weak
unramifiedness through the localization formalism of Gabber--Loeser.  The canonical comparison
morphisms \(\alpha_m\colon L^{\ast_!m}\to L^{\ast_*m}\), \(L=M\oplus M^\vee\), have Mellin transforms
with derived fibre \(\varphi_{L,\chi}^{\otimes m}\) at every character \(\chi\)
(Lemma~\ref{lem:comparison-monoidal}).  A Nakayama argument and \ref{item:GL5} invert them at every
character outside the locus \(\Delta(L)\) defined in \ref{item:GL2} of \S\ref{sec:mellin}.  The
compact and ordinary cohomology of every perverse cohomology sheaf of \(\Cone(\alpha_m)\) then
vanishes at such characters, and \ref{item:T5} gives unramifiedness.  This proves
Lemma~\ref{lem:weakly-unramified-inclusion}, the inclusion \(X_w(L)\subseteq X(M_0)\).  The
complement of \(X(M_0)\) is therefore contained in \(\Delta(L)(\ol\Ql)\), and
Theorem~\ref{thm:delta-count} applied to \(L\) bounds it.

Corollary~\ref{cor:equidistribution} follows from the Weyl criterion and \ref{item:T5}.  For every
nontrivial irreducible representation \(\rho\) of the common tannakian group,
Theorem~\ref{thm:stratified} bounds the characters at which the associated object has cohomology
outside degree zero, and Theorem~\ref{thm:unram-count} bounds the characters omitted from the
unramified averaging set.  The weight bounds for the remaining Frobenius traces then force every
nontrivial Weyl average to be \(O_\rho(q^{-1/2})\).

The text is organized as follows. \S\ref{sec:preliminaries} fixes the geometric and sheaf-theoretic conventions and recalls the complexity
estimates, Kummer sheaves, the Mellin transform, and the tannakian formalism used throughout.
\S\ref{sec:relative-mellin}
develops the relative Mellin transform in one coordinate, together with the complexity and
monodromy inputs over a base, and proves the exceptional-set bound.
\S\ref{sec:counting} proves Theorems~\ref{thm:delta-count} and~\ref{thm:stratified}, and
\S\ref{sec:unramified} proves Theorem~\ref{thm:unram-count} and
Corollary~\ref{cor:equidistribution}.

\subsection*{Acknowledgements}

We thank Nick Katz, Emmanuel Kowalski and Will Sawin for their comments on a first version of this
article.

\section{Preliminaries}
\label{sec:preliminaries}

This section fixes the geometric and sheaf-theoretic conventions and recalls the complexity estimates,
Kummer sheaves, the Mellin transform, and the tannakian formalism used throughout.

Throughout, \(k=\bbF_q\) is a finite field of characteristic \(p\), \(k_n=\bbF_{q^n}\), and \(\bar k\) is
a fixed algebraic closure of \(k\).  We fix a prime \(\ell\neq p\) and an algebraic closure \(\ol\Ql\) of
\(\mathbb Q_\ell\).  Unadorned varieties are over \(\bar k\), while arithmetic varieties carry a
displayed base field.  When a statement allows the algebraically closed base field to vary, we denote it
by \(\Omega\), and unadorned varieties in that statement and its proof are over \(\Omega\).

\subsection{Embeddings}
\label{sec:embeddings}

We write \(\Gm^d/k\) for the split \(d\)-dimensional torus and \(\Gm^d/\bar k\) for its geometric base
change.  Thus unadorned \(\Gm^d\) means \(\Gm^d/\bar k\), except in statements over \(\Omega\).
For every \(d\ge1\), we use the product compactification
\(\Gm^d\hookrightarrow(\bbP^1_{\bar k})^d\), with boundary induced by
\(\Gm\hookrightarrow\bbP^1\), and fix the locally closed immersion
\[
        u_d\colon \Gm^d\hookrightarrow\bbP^{2^d-1}_{\bar k}
\]
obtained by the Segre embedding.  For \(d\ge2\), these choices are compatible with the decomposition
\(\Gm^d=\Gm\times\Gm^{d-1}\), in that the embedding induced by \(u_1\) and \(u_{d-1}\) is \(u_d\).

For \(x\in\{0,\infty\}\), put
\[
 U_0=\bbP^1\setminus\{\infty\},\quad t_0(z)=z,
 \qquad
 U_\infty=\bbP^1\setminus\{0\},\quad t_\infty(z)=z^{-1}.
\]
Each \(t_x\) is an isomorphism \(U_x\xrightarrow{\ \sim\ }\bbA^1\) carrying \(x\) to \(0\).
We use the standard immersion \(a\colon\bbA^1\hookrightarrow\bbP^1\).  Products carry the corresponding
Segre embeddings.  The morphisms used below are products and compositions of coordinate projections,
multiplication, inversion, the identity section, the open and closed immersions associated with
\(\Gm\hookrightarrow\bbP^1\), the maps \(t_x\), and the Segre embeddings.  We equip each graph
with the induced locally closed immersion into the product of the projective ambient spaces.  These
embeddings and morphisms are fixed over \(\mathbb Z\), and will be used in 
Lemma~\ref{lem:standard-geometric-complexity}.

\subsection{Sheaf-theoretic conventions}
\label{sec:sheaves}

We use the pro-\'etale formalism of Bhatt--Scholze \cite{BhattScholze2015}.  For every scheme \(X\),
the pro-\'etale site \(X_{\mathrm{pro\acute et}}\) is defined, and every morphism
\(f\colon X\to Y\) induces a morphism of pro-\'etale topoi
\cite[Def.~4.1.1, Lem.~5.4.1]{BhattScholze2015}.  Consequently, \(Rf_*\) is defined on the ambient
derived categories without a finiteness hypothesis on \(f\).

All schemes in this paper are noetherian.  For a scheme \(Y\), we set
\[
        \Db(Y,\ol\Ql)=D_{\mathrm{cons}}(Y_{\mathrm{pro\acute et}},\ol\Ql),
\]
where the right-hand side is the constructible derived category of
\cite[Defs.~6.8.6, 6.8.8]{BhattScholze2015}.  It is equivalent to the usual constructible derived
category, by \cite[Prop.~6.8.14]{BhattScholze2015} together with
\cite[Def.~6.5.1, Prop.~6.6.11]{BhattScholze2015}.

Except in the construction of nearby cycles, the six operations in the main text are applied to schemes
of finite type over a field on which \(\ell\) is invertible.  They are quasi-compact and
quasi-separated, and every morphism to which \(Rf_!\) is applied is separated and of finite type,
hence finitely presented.  Thus the hypotheses of \cite[\S6.7, Rem.~6.8.15]{BhattScholze2015} are
satisfied.  The direct image \(R\bar j_{x,*}\) in the nearby-cycle construction is taken in the ambient
derived category.

For \(Y\) of finite type over either \(k\) or \(\bar k\), we write
\(\Perv(Y,\ol\Ql)\subset\Db(Y,\ol\Ql)\) for the abelian category of perverse sheaves with respect to the
middle perversity of \cite{BBDG18}, with perverse \(t\)-structure
\(({}^pD^{\le0},{}^pD^{\ge0})\).  Geometric base change from \(k\) to \(\bar k\) is \(t\)-exact.  We use
\(D\) for Verdier duality, \(\boxtimes\) for the external tensor product, \(\cH^i\) for the cohomology
sheaves of a complex, \(\Supp\) for support, \(\Cone\) for the mapping cone, and \(\HH^i\), \(\HH^i_c\)
for hypercohomology without and with compact supports.

Let \(f\colon X\to Y\) be a separated morphism of finite type.  Then \(X\) and \(Y\) are quasi-compact
and quasi-separated, and \(f\) is finitely presented, as required in
\cite[Def.~6.7.6]{BhattScholze2015}.  By Nagata compactification, choose a factorization
\(f=\bar f\circ j\), where
\(j\colon X\hookrightarrow\bar X\) is an open immersion and \(\bar f\colon\bar X\to Y\) is proper
\cite[Tag~0F41]{stacks-project}.  For an open immersion \(j\), let \(\varphi_j\colon j_!\to Rj_*\) be the
natural transformation whose value at \(K\) corresponds to \(\operatorname{id}_K\) under
\[
        \operatorname{Hom}(j_!K,Rj_*K)\simeq\operatorname{Hom}(K,K).
\]
The \emph{forget-supports transformation} for \(f\) is
\begin{equation}
\label{eq:forget-supports}
        \varphi_f=R\bar f_*(\varphi_j)\colon
        Rf_!=R\bar f_*\,j_!\longrightarrow R\bar f_*\,Rj_*=Rf_*.
\end{equation}
For \(K\in\Db(X,\ol\Ql)\), we write
\[
        \varphi_f(K)\colon Rf_!K\longrightarrow Rf_*K
\]
for the value of this natural transformation at \(K\).
It is independent of the compactification.  Its compatibilities with composition, external products, and
the lisse projection formula are recorded in Lemmas~\ref{lem:forget-supports}
and~\ref{lem:forget-supports-products}, proved in Appendix~\ref{app:forget-supports}.

We will use the following combined form of the compatibilities with composition and the lisse projection
formula.  Let \(p\colon X\to Y\) and \(f\colon Y\to S\) be separated morphisms of finite type, let
\(K\in\Db(X,\ol\Ql)\), and let \(\cL\) be lisse on \(Y\).  For \(\star\in\{!,*\}\), let
\[
 \beta_\star\colon
 R(f\circ p)_\star(K\otimes p^*\cL)
 \xrightarrow{\ \sim\ }
 Rf_\star(Rp_\star K\otimes\cL)
\]
be the composition and projection-formula isomorphism.
For \(\star=!\), the formula used here is the projection formula
\cite[Exp.~XVII, Prop.~5.2.9]{SGA4}, \cite[Prop.~A.1.5(ii)]{GabberLoeser1996}.  For \(\star=*\), it is
the universal coefficient formula for a lisse sheaf \cite[Prop.~A.1.5(i)]{GabberLoeser1996}.  A direct
proof of the latter is included in the proof of Lemma~\ref{lem:forget-supports-products}(2).
Then the following diagram commutes.
\begin{equation}
\label{eq:forget-supports-composition-lisse}
\begin{tikzcd}[column sep=large]
R(f\circ p)_!(K\otimes p^*\cL)
  \arrow[rr,"\varphi_{f\circ p}"]
  \arrow[d,"\beta_!"',"\rotatebox{90}{\(\sim\)}"]
&&
R(f\circ p)_*(K\otimes p^*\cL)
  \arrow[d,"\beta_*","\rotatebox{90}{\(\sim\)}"'] \\
Rf_!(Rp_!K\otimes\cL)
  \arrow[r,"Rf_!(\varphi_p(K)\otimes\operatorname{id}_{\cL})"']
&
Rf_!(Rp_*K\otimes\cL)
  \arrow[r,"\varphi_f"']
&
Rf_*(Rp_*K\otimes\cL).
\end{tikzcd}
\end{equation}
The composition formula in Lemma~\ref{lem:forget-supports}(2) gives the lower composite.  Lemma~\ref{lem:forget-supports-products}(2)
identifies its first arrow under the projection-formula isomorphisms.

For an inverse system of \'etale sheaves with surjective transition maps, Bhatt--Scholze identify
Jannsen's continuous \'etale cohomology with the pro-\'etale cohomology of its inverse limit
\cite[Prop.~5.6.2]{BhattScholze2015}.  We use this identification throughout.  In particular, the
comparison with Galois cohomology and the Hochschild--Serre spectral sequence
\cite[(3.2),(3.3)]{Jannsen88} apply.  For a profinite group \(G\), a finite extension
\(E/\mathbb Q_\ell\), and a finite-dimensional continuous \(E\)-representation \(V\) with a
\(G\)-stable \(\mathcal O_E\)-lattice \(\Lambda\), choose a uniformizer \(\varpi\) of \(E\) and set
\[
 R\Gamma_{\mathrm{cts}}(G,V)
 =\left(R\varprojlim_nR\Gamma_{\mathrm{cts}}(G,\Lambda/\varpi^n\Lambda)\right)
  \otimes_{\mathcal O_E}E.
\]
This is independent of \(\Lambda\) and \(\varpi\) and computes Tate's continuous-cochain cohomology
\cite[(2.1),(2.2)]{Jannsen88}.  For an \(\ol\Ql\)-representation, choose an \(E\)-model and extend
scalars from \(E\) to \(\ol\Ql\).  The result is independent of the model.  An \(E\)-model always
exists.  The image of a continuous representation \(G\to\operatorname{GL}_n(\ol\Ql)\) of a profinite
group is compact, the finite subextensions of \(\ol\Ql/\mathbb Q_\ell\) are countably many, and the
intersections of the image with the closed subgroups \(\operatorname{GL}_n(E)\) cover it.  By the Baire
category theorem, one such intersection is open in the image, and adjoining to \(E\) the entries of
finitely many coset representatives gives a finite extension containing all matrix entries.  A compact
subgroup of \(\operatorname{GL}_n(E)\) moreover stabilizes an \(\mathcal O_E\)-lattice.

\subsection{Complexity}
\label{sec:complexity}

For a complex valued function \(f\) and a nonnegative function \(g\), we write \(f\ll g\), or \(f=O(g)\), if
\(|f|\le Kg\) for some constant \(K\ge0\).  For nonnegative \(f\) we write \(f\asymp g\) if both
\(f\ll g\) and \(g\ll f\) hold.  A subscript on \(\ll\) or \(O\) lists parameters on which the constant
may depend, and any further dependence of constants is stated at each occurrence.

If an estimate first produces a function \(f:\mathbb Z_{\ge0}\to\mathbb R_{\ge0}\), we replace it by the
nondecreasing function \(f^\uparrow:\mathbb R_{\ge0}\to\mathbb R_{\ge0}\) defined by
\[
        f^\uparrow(C)=\max\{f(e)\mid e\in\mathbb Z_{\ge0},\ e\le\lceil C\rceil\}.
\]
We retain the notation \(f\) after this replacement.  Thus all functions used in upper bounds below are
nondecreasing.

Let \(\Omega\) be an algebraically closed field and let \(\ell\) be invertible in \(\Omega\).  We recall
Sawin's complexity and the properties used below.

\begin{enumerate}[label=\textup{(C\arabic*)}, ref=\textup{(C\arabic*)}]
\item\label{item:C1} \emph{Definition} \cite[Defs.~3.2, 6.3, 6.5, 6.6]{SFFK23}.  For
\(A\in\Db(\bbP^n_\Omega)\),
\[
        c(A)=\max_{0\le m\le n}\sum_{i\in\mathbb Z}
        h^i\bigl(\bbP^m_{\Omega'},l_{a_m}^*A\bigr),
\]
where \(a_m\) is a geometric generic point of the space of \((n+1)\times(m+1)\) matrices, defined over
an algebraically closed extension \(\Omega'\) of \(\Omega\), and
\(l_{a_m}:\bbP^m_{\Omega'}\to\bbP^n_{\Omega'}\) is the associated linear map.

For an embedded quasi-projective variety \(u:X\hookrightarrow\bbP^n\) and \(A\in\Db(X)\), put
\(c_u(A)=c(u_!A)\) and \(c(u)=c_u(\ol\Ql)\).
For a morphism \(f:(X,u)\to(Y,v)\) with \(u:X\hookrightarrow\bbP^m\) and \(v:Y\hookrightarrow\bbP^n\),
let \(a_p\) and \(b_q\) be geometric generic points of the spaces of \((m+1)\times(p+1)\) and
\((n+1)\times(q+1)\) matrices, defined over a common algebraically closed extension \(\Omega'\) of
\(\Omega\).  The complexity of \(f\) relative to \((u,v)\) is
\[
        c_{u,v}(f)=\max_{0\le p\le m}\max_{0\le q\le n}\sum_{i\in\mathbb Z}
        h^i_c\bigl(X_{\Omega'},\,u^*l_{a_p,*}\ol\Ql\otimes f^*v^*l_{b_q,*}\ol\Ql\bigr).
\]
The Betti numbers may equally be computed as those of the intersection of the graph of \(f\) in
\(\bbP^m\times\bbP^n\) with a product of linear subspaces, for \((a_p,b_q)\) in a dense open subset of
the product of the two matrix spaces.  On
\(X=\operatorname{Spec}\Omega\), the complexity of a complex is the sum of the dimensions of its
cohomology groups.  For \(A\in\Db(\Gm^d)\), we write
\[
        \cu(A)=c_{u_d}(A)=c(u_{d,!}A).
\]

\item\label{item:C2} \emph{Positivity} \cite[Prop.~3.4, Rem.~6.4]{SFFK23}.  The complexity
\(c_u(A)\) is a nonnegative integer, and \(c_u(A)=0\) if and only if \(A=0\).

\item\label{item:C3} \emph{Six operations} \cite[Th.~6.8, estimates~(6.1), (6.2), (6.5), (6.6), (6.8)]{SFFK23}.
For a morphism \(f:(X,u)\to(Y,v)\), objects \(A,B\in\Db(X)\), and \(C\in\Db(Y)\),
\[
\begin{gathered}
        c_u(DA)\ll c(u)\,c_u(A),
        \qquad
        c_u(A\otimes B)\ll c_u(A)\,c_u(B),
        \qquad
        c_u(f^*C)\ll c_{u,v}(f)\,c_v(C),\\
        c_v(Rf_!A)\ll c_{u,v}(f)\,c_u(A),
        \qquad
        c_v(Rf_*A)\ll c(u)\,c(v)\,c_{u,v}(f)\,c_u(A),
\end{gathered}
\]
with implied constants depending only on the ambient projective dimensions.

\item\label{item:C4} \emph{Triangles, sums, shifts} \cite[Prop.~6.14]{SFFK23}.  In a distinguished
triangle, the complexity of each vertex is at most the sum of the complexities of the other two.
Moreover \(c_u(A\oplus B)=c_u(A)+c_u(B)\), and \(c_u(A[h])=c_u(A)\) for \(h\in\mathbb Z\).  In
particular, a direct summand has complexity at most that of the ambient object.

\item\label{item:C5} \emph{External products} \cite[Prop.~6.12]{SFFK23}.  For embedded varieties
\((X,u)\) and \((Y,v)\), with \(u\boxtimes v\) the embedding of \(X\times Y\) induced by the Segre
embedding and \(p_1,p_2\) the projections,
\[
        c_{u\boxtimes v}(A\boxtimes B)
        \ll c_{u\boxtimes v,u}(p_1)\,c_{u\boxtimes v,v}(p_2)\,c_u(A)\,c_v(B),
\]
with an implied constant depending only on the ambient projective dimensions.

\item\label{item:C6} \emph{Jordan--H\"older} \cite[Th.~6.15]{SFFK23}.  If
\((A_{i,j})_{1\le j\le n_i}\) are the Jordan--H\"older factors of \({}^pH^i(A)\), repeated with
multiplicity, then
\[
        c_u(A)\le\sum_{i\in\mathbb Z}\sum_{j=1}^{n_i}c_u(A_{i,j})\ll c(u)\,c_u(A),
\]
with an implied constant depending only on the ambient projective dimension.  Combined with
\ref{item:C4}, this shows that every subquotient of every \({}^pH^i(A)\) has complexity at most a
constant multiple of \(c(u)\,c_u(A)\).

\item\label{item:C7} \emph{Nearby and vanishing cycles} \cite[Cor.~6.18]{SFFK23}.  Let
\(f:(X,u)\to(S,v)\) be flat, with \(S\) a smooth irreducible curve, and let \(\sigma\in S\) be closed.
Put \(S_{(\sigma)}=\operatorname{Spec}\mathcal O^{\mathrm{sh}}_{S,\sigma}\), and let
\[
        f_{(\sigma)}:X\times_S S_{(\sigma)}\longrightarrow S_{(\sigma)}
\]
be the base change of \(f\).  Denote by \(\Psi_f\) and \(\Phi_f\) the nearby- and vanishing-cycle
functors for \(f_{(\sigma)}\), whose values on \(A\in\Db(X)\) are formed from the pullback of \(A\) to
\(X\times_S S_{(\sigma)}\) and viewed as objects on \(X_\sigma\).  If \(u_\sigma\) denotes the induced
embedding of \(X_\sigma\), then
\[
        c_{u_\sigma}(\Psi_fA)\ll c_{u,v}(f)^2c_u(A),
        \qquad
        c_{u_\sigma}(\Phi_fA)\ll c_{u,v}(f)^2c_u(A).
\]
The implied constants depend only on the ambient projective dimensions.
\end{enumerate}

The implied constants in \ref{item:C3}, \ref{item:C5}, \ref{item:C6} and \ref{item:C7} are uniform in
\(\Omega\) and in \(\ell\), while \(c(u)\) and \(c_{u,v}(f)\) remain inputs.\footnote{Uniformity in
\(\ell\) means that one constant serves every prime \(\ell\) invertible in \(\Omega\).  This holds
because the implied constants in the statements cited in \ref{item:C3}, \ref{item:C5},
\ref{item:C6} and \ref{item:C7} depend only on the ambient projective dimensions.}

\begin{lem}
\label{lem:betti-presentation}
For all integers \(n_1,n_2,r,e\ge0\) there is a constant \(B_{n_1,n_2,r,e}\ge1\), independent of
\(\Omega\) and of \(\ell\), with the following property.  Let \(Z=Z_1\setminus Z_2\), where \(Z_1\) and
\(Z_2\) are closed subschemes of \(\bbP^{n_1}_\Omega\times\bbP^{n_2}_\Omega\) defined, on each product
of standard affine charts, by at most \(r\) equations of degrees at most \(e\).  Then every compactly
supported Betti sum of a generic linear or product-linear section of \(Z\), as in
\cite[Defs.~6.3, 6.6]{SFFK23}, is at most \(B_{n_1,n_2,r,e}\).
\end{lem}

\begin{proof}
A generic linear or product-linear section imposes additional linear equations and preserves the chart
presentation.  A finite cover by principal affine opens, realized as closed affine schemes by adjoining
inverse variables, together with the \v{C}ech spectral sequence reduces the required sums to those of
closed affine schemes defined by a number of equations and degrees bounded in terms of \(n_1\),
\(n_2\), \(r\), and \(e\).  Passing to the reduction does not change \'etale cohomology.
Proposition~6.21 of \cite{SFFK23}, applied with target a point, bounds these sums by a constant
depending only on the ambient dimensions and on the number and degrees of the defining equations.
\end{proof}

\begin{lem}
\label{lem:standard-geometric-complexity}
For every \(d\ge1\), there is an integer \(D_d^{\mathrm{std}}\ge1\), depending only on \(d\), such that
the complexities of the embeddings and embedded morphisms specified in \S\ref{sec:embeddings} that occur
in dimensions at most \(d\), including their products with \((\Gm^r,u_r)\) for \(1\le r\le d\), are at
most \(D_d^{\mathrm{std}}\) after base change to any algebraically closed field \(\Omega\).  The integer
\(D_d^{\mathrm{std}}\) is independent of every prime \(\ell\) invertible in \(\Omega\).
\end{lem}

\begin{proof}
For fixed \(d\), only finitely many embedded schemes and morphisms occur.  By
\cite[Def.~6.6]{SFFK23}, it suffices to bound the compactly supported Betti sums of generic linear
sections of these schemes and generic product-linear sections of their graphs.  On standard affine
charts, the schemes, graphs, and boundaries are defined by a number of equations and degrees bounded in
terms of \(d\).  Lemma~\ref{lem:betti-presentation} bounds these sums, and taking the maximum gives
\(D_d^{\mathrm{std}}\).
\end{proof}

\begin{lem}
\label{lem:linear-reembedding}
For all integers \(0\le N\le M\), there is a constant \(E_{N,M}\ge1\), independent of \(\Omega\) and of
\(\ell\), with the following property.  For every locally closed embedding
\(w:Y\hookrightarrow\bbP^N_\Omega\), every closed immersion
\(l:\bbP^N_\Omega\hookrightarrow\bbP^M_\Omega\) induced by an injective linear map
\(\Omega^{N+1}\to\Omega^{M+1}\), and every \(B\in\Db(Y,\ol\Ql)\),
\[
        E_{N,M}^{-1}\,c_w(B)\le c_{l\circ w}(B)\le E_{N,M}\,c_w(B).
\]
\end{lem}

\begin{proof}
View \(l\) as a morphism \((\bbP^N_\Omega,\operatorname{id})\to(\bbP^M_\Omega,\operatorname{id})\).  In the
formula of \ref{item:C1} for the complexity of a morphism, \(l_{a_p,*}\ol\Ql\) is the constant sheaf
on a linear subspace of \(\bbP^N_{\Omega'}\), and proper base change identifies \(l^*l_{b_q,*}\ol\Ql\)
with the constant sheaf on \(l^{-1}(\Pi)\), where \(\Pi\subseteq\bbP^M_{\Omega'}\) is the image of
\(l_{b_q}\).  Their tensor product is the constant sheaf on the intersection, again a linear subspace
of \(\bbP^N_{\Omega'}\), whose compactly supported Betti sum is at most \(N+1\).  Hence
\(c_{\operatorname{id},\operatorname{id}}(l)\le N+1\).  Since \(l\) is a closed immersion, one has
\(c_{l\circ w}(B)=c\bigl(l_!(w_!B)\bigr)\) and \(w_!B=l^*l_!(w_!B)\).  The pushforward and pullback
estimates \ref{item:C3} applied to \(l\) bound each of \(c_w(B)\) and \(c_{l\circ w}(B)\) by
\((N+1)\) times the other, up to implied constants depending only on \(N\) and \(M\).  Taking
\(E_{N,M}\) to dominate both products gives the claim.
\end{proof}

\begin{lem}
\label{lem:lissity-complexity}
For every \(N\ge0\), there is a nondecreasing function
\[
        H_N:\mathbb R_{\ge0}\longrightarrow\mathbb R_{\ge0}
\]
with the following property.  Let \(w:Y\hookrightarrow\bbP^N_\Omega\) be a locally closed embedding and
let \(K\in\Db(Y,\ol\Ql)\) satisfy \(c_w(K)\le C\).  There is a stratification
\[
        \Supp K=\bigsqcup_{\alpha\in I}W_\alpha
\]
by smooth connected locally closed subvarieties such that every \(\cH^i(K)|_{W_\alpha}\) is lisse and
\[
        \#I\le H_N(C),
        \qquad
        \sum_i\operatorname{rk}\bigl(\cH^i(K)|_{W_\alpha}\bigr)\le H_N(C)
        \quad(\alpha\in I).
\]
The function \(H_N\) is independent of \(\Omega\) and of \(\ell\).
\end{lem}

\begin{proof}
For integral \(C\), Lemma~6.26 of \cite{SFFK23} gives a lissity stratification with bounds depending only
on \((C,N)\) for the degrees of its components and the complexities of their immersions.  Refining into
connected components preserves such a bound.  The restriction estimate \ref{item:C3} and the
generic-rank bound \cite[Prop.~3.4]{SFFK23} then bound the displayed rank sum on every stratum.  Taking the
maximum of these bounds and applying the nondecreasing-majorant convention defines \(H_N\).
\end{proof}

For each \(N\ge0\), fix a function \(H_N\) as in Lemma~\ref{lem:lissity-complexity}.

\begin{lem}
\label{lem:eigenvalue-strata}
Let \(w:Y\hookrightarrow\bbP^N_\Omega\) be a locally closed embedding, let
\(K\in\Db(Y,\ol\Ql)\), and let \(\tau\in\operatorname{Aut}_{\Db(Y)}(K)\).  Put
\[
        \Eig_Y(K,\tau)
        =
        \bigcup_{\substack{i\in\mathbb Z\\ \bar y\to Y\ \mathrm{geometric}}}
        \Eig\bigl(\cH^i(\tau)_{\bar y}\bigr),
\]
where \(\Eig\bigl(\cH^i(\tau)_{\bar y}\bigr)\) denotes the set of eigenvalues on
\(\cH^i(K)_{\bar y}\).  Then
\[
        \#\Eig_Y(K,\tau)\le H_N(c_w(K))^2.
\]
\end{lem}

\begin{proof}
Choose a stratification as in Lemma~\ref{lem:lissity-complexity}.  On each connected stratum, the
characteristic polynomial of \(\cH^i(\tau)_{\bar y}\) is independent of \(\bar y\) for every \(i\).  The
number of eigenvalues contributed by the stratum is at most the sum of the ranks of the \(\cH^i(K)\).
Lemma~\ref{lem:lissity-complexity} bounds this sum and the number of strata by \(H_N(c_w(K))\).
\end{proof}

\subsection{Tame ramification and Kummer sheaves}
\label{sec:tame-kummer}

Let \(T/\Omega\) be a torus over an algebraically closed field.  Choose a smooth compactification
\(T\hookrightarrow\overline T\) whose boundary is a normal-crossings divisor.  SGA~1 defines tame
ramification along the boundary at its generic points and defines \(\pi_1^{\mathrm t}(T)\) as the quotient
of \(\pi_1(T)\) classifying finite \'etale covers with this ramification property
\cite[Exp.~XIII, 2.1.1, 2.1.3, and Rem.~2.3(c)]{SGA1}.  Abhyankar's lemma shows that this condition is
independent of the choice of a smooth compactification with normal-crossings boundary
\cite[Exp.~XIII, Prop.~5.2]{SGA1}, \cite[\S1.9]{GabberLoeser1996}.  Consequently, the resulting quotient
\(\pi_1^{\mathrm t}(T)\) does not depend on the compactification.  When \(\Omega\) has positive
characteristic, Gabber--Loeser call the rank-one lisse sheaf attached to a continuous character of
\(\pi_1^{\mathrm t}(T)\) a Kummer sheaf \cite[\S\S1.3--1.4]{GabberLoeser1996}.

When \(\operatorname{char}\Omega=p>0\), put
\(\widehat{\mathbb Z}^{(p')}=\prod_{\ell'\ne p}\mathbb Z_{\ell'}\), where the product is over primes
\(\ell'\ne p\), and in characteristic zero put
\(\widehat{\mathbb Z}^{(p')}=\widehat{\mathbb Z}\).  For \(n\) invertible in \(\Omega\), the isogeny
\([n]:T\to T\) is a connected Galois cover with group \(T[n]=X_*(T)\otimes\mu_n\), where
\(X_*(T)=\operatorname{Hom}(\Gm,T)\) is the cocharacter lattice.  Kummer theory identifies
\(\pi_1^{\mathrm t}(T,1)\) with the limit of these Galois groups over the isogenies \([n]\) for which
\(n\) is invertible in \(\Omega\), giving a canonical isomorphism
\[
        \pi_1^{\mathrm t}(T,1)
        \simeq X_*(T)\otimes_{\mathbb Z}\widehat{\mathbb Z}^{(p')}(1)
\]
\cite[\S1.3]{GabberLoeser1996}.  In characteristic zero the tame fundamental group is the full \'etale
fundamental group and the same identification is classical.  The identification is functorial in
homomorphisms of tori, since \(f\circ[n]=[n]\circ f\) induces \(X_*(f)\otimes\operatorname{id}_{\mu_n}\)
on the Galois groups.  The coordinate projections identify \(X_*(\Gm^d)\) with
\(\bigoplus_{j=1}^dX_*(\Gm)\), so they induce an isomorphism
\begin{equation}
\label{eq:tame-product}
        \pi_1^{\mathrm t}(\Gm^d/\Omega,1)
        \xrightarrow{\ \sim\ }
        \prod_{j=1}^d\pi_1^{\mathrm t}(\Gm/\Omega,1).
\end{equation}
In particular, \(\pi_1^{\mathrm t}(\Gm^d,1)\) is abelian.  In positive characteristic \(p\), it is
pro-prime-to-\(p\).

We now return to the standard compactification \(j:\Gm/\bar k\hookrightarrow\bbP^1_{\bar k}\).  For
\(x\in\{0,\infty\}\), write
\(\bbP^1_{(x)}=\operatorname{Spec}\mathcal O^{\mathrm{sh}}_{\bbP^1,x}\), with generic point
\(\eta_x\) and geometric generic point \(\bar\eta_x\).  Put
\[
        I_x=\operatorname{Gal}(\bar\eta_x/\eta_x),
        \qquad
        I_x^{\mathrm t}=I_x/P_x,
\]
where the wild inertia subgroup \(P_x\) is the kernel of the action of \(I_x\) on the
maximal tamely ramified extension of the fraction field of
\(\mathcal O^{\mathrm{sh}}_{\bbP^1,x}\).  Applied to the trait \(\bbP^1_{(x)}\) with the
divisor \(\{x\}\), \cite[Exp.~XIII, Cor.~5.3]{SGA1} identifies \(I_x^{\mathrm t}\)
canonically with \(\widehat{\mathbb Z}^{(p')}(1)\), through the action on compatible systems of
prime-to-\(p\) roots of a uniformizer.  Choose a topological generator \(\tau_x\) of
\(I_x^{\mathrm t}\).  The primary decomposition of \(\widehat{\mathbb Z}^{(p')}(1)\) gives a canonical
product decomposition
\[
        I_x^{\mathrm t}=J_x\times\Gamma_x,
        \qquad
        \tau_x=(\tau_{x,J},\gamma_x),
\]
where \(J_x\) is the product of the pro-\(\ell'\) factors over the primes \(\ell'\neq p,\ell\).  The
factor \(\Gamma_x\) is isomorphic to \(\mathbb Z_\ell(1)\), and \(\gamma_x\) is a topological generator
of \(\Gamma_x\).

The scheme map \(\eta_x\to\Gm\) induces
\(I_x\to\pi_1(\Gm,\bar\eta_x)\) and, after a choice of path from \(\bar\eta_x\) to
\(1\), a homomorphism \(I_x\to\pi_1(\Gm,1)\).  Its composite with the projection to
\(\pi_1^{\mathrm t}(\Gm,1)\) is independent of the path, because the target is abelian, and factors
through \(I_x^{\mathrm t}\), because the image of the pro-\(p\) group \(P_x\) in a
pro-prime-to-\(p\) group is trivial.  For a continuous character \(\lambda\) of
\(\pi_1^{\mathrm t}(\Gm,1)\), let \(\lambda_x\) denote the restriction of \(\lambda\) to
\(I_x^{\mathrm t}\) along this homomorphism.

\begin{lem}
\label{lem:kummer-local-monodromy}
For each \(x\in\{0,\infty\}\), the homomorphism
\[
        I_x^{\mathrm t}\longrightarrow\pi_1^{\mathrm t}(\Gm,1)
\]
just defined is an isomorphism.  Consequently, restriction of characters gives an isomorphism
\[
        \operatorname{Hom}_{\mathrm{cts}}\bigl(\pi_1^{\mathrm t}(\Gm,1),\ol\Ql^\times\bigr)
        \xrightarrow{\ \sim\ }
        \operatorname{Hom}_{\mathrm{cts}}(I_x^{\mathrm t},\ol\Ql^\times),
\]
and the map \(\lambda\mapsto\lambda_x(\tau_x)\) is injective on continuous characters of
\(\pi_1^{\mathrm t}(\Gm,1)\).
\end{lem}

\begin{proof}
Apply \cite[Exp.~XIII, Cor.~2.12]{SGA1} to
\(\Gm=\bbP^1\setminus\{0,\infty\}\).  In the local identifications of
\cite[Exp.~XIII, Cor.~5.3]{SGA1}, use the uniformizers \(t\) at \(0\) and \(t^{-1}\) at \(\infty\).
With respect to these identifications and the global Kummer identification above, the maps at \(0\) and
\(\infty\) are the identity and inversion, respectively, on \(\widehat{\mathbb Z}^{(p')}(1)\).  Hence
both maps are isomorphisms.  The assertions about characters follow because a continuous character of a
procyclic group is determined by its value at a topological generator.
\end{proof}

\begin{lem}
\label{lem:kummer-complexity}
For every \(d\ge1\), there is a constant \(\kappa_d\), depending only on \(d\), such that
\[
        \cu(\cL)\le\kappa_d
\]
for every algebraically closed field \(\Omega\), every prime \(\ell\) invertible in \(\Omega\), and every
rank-one lisse \(\ol\Ql\)-sheaf \(\cL\) on \(\Gm^d/\Omega\) whose monodromy character factors through
\(\pi_1^{\mathrm t}(\Gm^d/\Omega,1)\).
\end{lem}

\begin{proof}
Let \(\lambda\) be the monodromy character of \(\cL\).  The product formula \eqref{eq:tame-product}
decomposes \(\lambda\) as \(\prod_{j=1}^d\lambda_j\circ\operatorname{pr}_j\), where each \(\lambda_j\)
is a continuous character of \(\pi_1^{\mathrm t}(\Gm/\Omega,1)\).  Hence
\(\cL\simeq\boxtimes_{j=1}^d\cL_{\lambda_j}\).

Suppose that \(d=1\) and that \(\cL\) is nontrivial.  A lisse sheaf on the connected nonproper curve
\(\Gm\) has no nonzero sections with finite support, so \(H^0_c(\Gm,\cL)=0\).  Poincar\'e duality
identifies \(H^2_c(\Gm,\cL)\) with a Tate twist of the dual of \(H^0(\Gm,\cL^\vee)\), which vanishes
because \(\cL^\vee\) is nontrivial.  The restrictions of \(\lambda\) to the inertia groups \(I_0\) and
\(I_\infty\) are tamely ramified.  Both Swan conductors of \(\cL\) therefore vanish, and the
Grothendieck--Ogg--Shafarevich formula \cite[\textit{Sommes trigonom\'etriques}, (3.2.1)]{SGA45}
gives \(\chi_c(\Gm,\cL)=0\), so \(H^1_c(\Gm,\cL)=0\) as well.  The curve-complexity formula
\cite[Th.~7.3(1)]{SFFK23}, applied to the degree-one embedding \(u_1\), gives \(\cu(\cL)=1\), and it
gives \(\cu(\ol\Ql)=2\) for the trivial local system.

For general \(d\), the external-product estimate \ref{item:C5}, together with the uniform bounds of
Lemma~\ref{lem:standard-geometric-complexity}, bounds
\(\cu(\boxtimes_{j=1}^d\cL_{\lambda_j})\) by a constant \(\kappa_d\).
\end{proof}

\subsection{Finite-field characters and products}
\label{sec:characters}

We retain the notation \(\widehat{\Gm^d}\), \(\cL_\chi\), \(A_\chi\),
\(\widehat{\Gm^d(k_n)}\), and \(Z(k_n)\) introduced in \S\ref{sec:statements}.  The character scheme
is that constructed by Gabber--Loeser \cite[\S3.2]{GabberLoeser1996} and used by
Forey--Fres\'an--Kowalski \cite{ForeyFresanKowalski2021}.  Its \(\ol\Ql\)-points are, by construction,
the continuous characters of \(\pi_1^{\mathrm t}(\Gm^d,1)\).  We identify a character with the
corresponding \(\ol\Ql\)-point of the character scheme.  The associated sheaf \(\cL_\chi\) therefore has
monodromy factoring through \(\pi_1^{\mathrm t}(\Gm^d,1)\), and hence is tamely ramified along the
boundary of \(\Gm^d\hookrightarrow(\bbP^1)^d\).

Let \(F\) denote the \(q\)-power Frobenius endomorphism of \(\Gm^d/k\).  After base change to \(k_n\), the
Lang isogeny is
\[
        L_n=F^n\cdot\operatorname{inv}\colon\Gm^d/k_n\longrightarrow\Gm^d/k_n,
        \qquad
        (x_1,\ldots,x_d)\longmapsto(x_1^{q^n-1},\ldots,x_d^{q^n-1}).
\]
It is a finite \'etale Galois covering with group \(\Gm^d(k_n)\).  It attaches to each
\(\chi\in\widehat{\Gm^d(k_n)}\) a Kummer sheaf \(\cL_{\chi,0}\) on \(\Gm^d/k_n\) whose \(k_n\)-trace
function is \(\chi\).  Its base change to \(\Gm^d/\bar k\) is the Kummer sheaf \(\cL_\chi\).  The Lang
covering remains connected after base change to \(\bar k\), so the geometric fundamental group of
\(\Gm^d\) surjects onto \(\Gm^d(k_n)\).  Geometric base change therefore defines an injective homomorphism
\[
        \widehat{\Gm^d(k_n)}\lhook\joinrel\longrightarrow\widehat{\Gm^d}(\ol\Ql),
        \qquad
        \chi\longmapsto[\cL_\chi].
\]

For \(a,b\ge1\), the standard product decomposition gives
\[
        \widehat{\Gm^{a+b}(k_n)}
        =\widehat{\Gm^a(k_n)}\times\widehat{\Gm^b(k_n)},
        \qquad
        |\widehat{\Gm(k_n)}|=q^n-1 .
\]
For characters \(\lambda,\eta\) of the two factors,
\(\cL_{(\lambda,\eta)}\simeq\operatorname{pr}_1^*\cL_\lambda\otimes\operatorname{pr}_2^*\cL_\eta\).

\subsection{Negligible objects}
\label{sec:negligible}

Following \cite[\S3.6]{GabberLoeser1996}, a perverse sheaf \(N\) on \(\Gm^d\) is \emph{negligible} if
\(\chi(\Gm^d,N)=0\).\footnote{Forey--Fres\'an--Kowalski define negligibility by requiring that the set
\(\bigl\{\chi\in\widehat{\Gm^d}(\ol\Ql)\mid\HH^0(\Gm^d,N_\chi)=0\bigr\}\) be generic in the sense of
\S\ref{sec:equidistribution} \cite[Def.~3.4]{ForeyFresanKowalski2021}, for perverse sheaves
defined over a finite extension of \(k\).  For such \(N\) the two definitions agree.  For generic
\(\lambda\), generic vanishing and invariance of the Euler--Poincar\'e characteristic under Kummer
twists give \(\dim\HH^0(\Gm^d,N_\lambda)=\chi(\Gm^d,N)\).  See
\cite[Prop.~3.22(2)--(3)]{ForeyFresanKowalski2021}.}  Let
\(\Perv_{\mathrm{neg}}(\Gm^d)\) be the Serre subcategory of negligible perverse sheaves
\cite[Lem.~3.6.2]{GabberLoeser1996}.  A complex \(K\in\Db(\Gm^d)\) is \emph{negligible} if
\({}^pH^i(K)\) is negligible for every \(i\).  The negligible complexes form a thick triangulated
subcategory of \(\Db(\Gm^d)\)
\cite[Prop.~3.6.1(i)]{GabberLoeser1996}.

For a perverse sheaf \(P\), put
\[
 \mathcal N(P)=\bigl\{\chi\in\widehat{\Gm^d}(\ol\Ql):
 R\Gamma_c(\Gm^d,P_\chi)=0\ \text{and}\ R\Gamma(\Gm^d,P_\chi)=0\bigr\}.
\]
For a complex \(K\), put
\[
        \mathcal N(K)=\bigcap_{i\in\mathbb Z}\mathcal N({}^pH^iK).
\]
A character \(\chi\) lies in \(\mathcal N(K)\) exactly when both cohomologies of every perverse
cohomology sheaf of \(K\) vanish at \(\chi\).\footnote{The corresponding definition in
\cite[\S3.4]{ForeyFresanKowalski2021}, stated for negligible complexes, displays a union in place of
this intersection.  The union is
a misprint.  It would make \(\mathcal N(K)\) the full character set, because \({}^pH^iK\) vanishes
for all but finitely many \(i\) and \(\mathcal N(0)\) is the full character set.  The proof of
\cite[Lem.~3.28]{ForeyFresanKowalski2021} uses the intersection condition.}

Following \cite[Def.-Prop.~3.7.2]{GabberLoeser1996}, let \(\Perv_{\mathrm{int}}(\Gm^d)\) be the full
subcategory of perverse sheaves with no nonzero negligible subobject or quotient.  Localization gives an
equivalence
\(\Perv_{\mathrm{int}}(\Gm^d)\xrightarrow{\sim}\Perv(\Gm^d)/\Perv_{\mathrm{neg}}(\Gm^d)\)
\cite[Def.-Prop.~3.7.2]{GabberLoeser1996}.

Following \cite[\S3.8]{ForeyFresanKowalski2021}, let
\(\Perv_{\mathrm{neg}}^{\mathrm{ari}}(\Gm^d/k)\) and
\(\Perv_{\mathrm{int}}^{\mathrm{ari}}(\Gm^d/k)\) be the full subcategories of
\(\Perv(\Gm^d/k,\ol\Ql)\) consisting of the \(M_0\) for which \(M_{0,\bar k}\) belongs to
\(\Perv_{\mathrm{neg}}(\Gm^d)\) and \(\Perv_{\mathrm{int}}(\Gm^d)\), respectively.  The first is a Serre
subcategory.  Put
\[
        \mathcal P^{\mathrm{ari}}(\Gm^d/k)
        =\Perv(\Gm^d/k)/\Perv_{\mathrm{neg}}^{\mathrm{ari}}(\Gm^d/k),
\]
and let \(q^{\mathrm{ari}}\) be the localization functor.  It induces an equivalence
\[
        \Perv_{\mathrm{int}}^{\mathrm{ari}}(\Gm^d/k)
        \xrightarrow{\ \sim\ }
        \mathcal P^{\mathrm{ari}}(\Gm^d/k).
\]

\subsection{Convolution}
\label{sec:convolution}

Let \(T\) denote either \(\Gm^d/k\) or \(\Gm^d/\bar k\), and let
\(m_T\colon T\times T\to T\) be multiplication.  For \(A,B\in\Db(T)\), we write
\(A\ast_!B=R(m_T)_!(A\boxtimes B)\) and \(A\ast_*B=R(m_T)_*(A\boxtimes B)\).  The transformation
\eqref{eq:forget-supports} for \(m_T\) gives the canonical morphism
\begin{equation}
\label{eq:convolution-comparison}
        c_{A,B}=\varphi_{m_T}(A\boxtimes B)\colon A\ast_!B\longrightarrow A\ast_*B.
\end{equation}
For \(T=\Gm^d/\bar k\), its twist by \(\cL_\chi\) is denoted \(c_{A,B,\chi}\).
If \(K\) is negligible, then \(K\ast_!A\) and \(K\ast_*A\) are negligible for every
\(A\in\Db(\Gm^d)\).  Moreover, \(\Cone(c_{A,B})\) is negligible for all \(A,B\in\Db(\Gm^d)\).  If \(A\)
and \(B\) are perverse, then
\({}^pH^i(A\ast_!B)\) and \({}^pH^i(A\ast_*B)\) are negligible for every \(i\neq0\)
\cite[Prop.~3.6.4]{GabberLoeser1996}.

Let \(\Db(\Gm^d)_{\mathrm{loc}}\) be the Verdier quotient of \(\Db(\Gm^d)\) by the full thick triangulated
subcategory of negligible complexes.  The perverse \(t\)-structure descends to this quotient, and its
heart is canonically equivalent to \(\Perv(\Gm^d)/\Perv_{\mathrm{neg}}(\Gm^d)\)
\cite[Prop.~3.6.1]{GabberLoeser1996}.  The products \(\ast_!\) and \(\ast_*\) induce the same
\(t\)-biexact product \(\ast_{\mathrm{loc}}\) on \(\Db(\Gm^d)_{\mathrm{loc}}\)
\cite[Def.-Prop.~3.6.5]{GabberLoeser1996}.  For perverse \(A\) and \(B\), this product in the heart is
represented by the localization of either \({}^pH^0(A\ast_!B)\) or \({}^pH^0(A\ast_*B)\).

Transporting this product through the equivalence of \S\ref{sec:negligible} defines the
\emph{internal convolution} \(\ast_{\mathrm{int}}\) on \(\Perv_{\mathrm{int}}(\Gm^d)\).\footnote{The
name follows \cite[Def.~3.11]{ForeyFresanKowalski2021}.  Gabber--Loeser call this product the
intermediate convolution, and the remark closing \cite[\S3.7]{GabberLoeser1996} identifies it for
\(d=1\) with Katz's middle convolution.}  For a perverse
sheaf \(P\), let \(P_{\mathrm{int}}\) denote the image of its localization under the quasi-inverse
equivalence.  For \(A,B\in\Perv_{\mathrm{int}}(\Gm^d)\),
\[
        A\ast_{\mathrm{int}}B
        \simeq\bigl({}^pH^0(A\ast_!B)\bigr)_{\mathrm{int}}
        \simeq\bigl({}^pH^0(A\ast_*B)\bigr)_{\mathrm{int}}
\]
canonically \cite[Def.-Prop.~3.7.3]{GabberLoeser1996}.

Following \cite[\S3.8]{ForeyFresanKowalski2021}, let
\(A_0,B_0\in\Perv(\Gm^d/k,\ol\Ql)\).  The two objects
\[
 q^{\mathrm{ari}}\bigl({}^pH^0(A_0\ast_!B_0)\bigr)
 \quad\text{and}\quad
 q^{\mathrm{ari}}\bigl({}^pH^0(A_0\ast_*B_0)\bigr)
\]
depend functorially only on \(q^{\mathrm{ari}}(A_0)\) and \(q^{\mathrm{ari}}(B_0)\), and they are
canonically isomorphic.  Their common product makes
\(\mathcal P^{\mathrm{ari}}(\Gm^d/k)\) a rigid symmetric \(\ol\Ql\)-linear tensor category.
Transporting it through the arithmetic equivalence in \S\ref{sec:negligible} defines
\(\ast_{\mathrm{int}}\) on \(\Perv_{\mathrm{int}}^{\mathrm{ari}}(\Gm^d/k)\).  Geometric base change is a
symmetric tensor functor
\[
 \bigl(\Perv_{\mathrm{int}}^{\mathrm{ari}}(\Gm^d/k),\ast_{\mathrm{int}}\bigr)
 \longrightarrow
 \bigl(\Perv_{\mathrm{int}}(\Gm^d),\ast_{\mathrm{int}}\bigr)
\]
\cite[\S3.8]{ForeyFresanKowalski2021}.  We use \(\ast_{\mathrm{int}}\) for both internal convolution
products, with the base determined by the category containing the objects.  In either category, the
convolution dual is \(P^\vee=\operatorname{inv}^*D(P)\) \cite[\S1.5]{ForeyFresanKowalski2021}.

For \(?\in\{!,*\}\) and \(A\in\Db(T)\), and for \(?=\mathrm{int}\) and \(A\) in either
\(\Perv_{\mathrm{int}}(\Gm^d)\) or \(\Perv_{\mathrm{int}}^{\mathrm{ari}}(\Gm^d/k)\), define the convolution
powers recursively by
\[
        A^{\ast_?1}=A,
        \qquad
        A^{\ast_?(r+1)}=A^{\ast_?r}\ast_?A.
\]

\begin{cor}
\label{cor:convolution-complexity}
For every \(d,r\ge1\) and \(?\in\{!,*,\mathrm{int}\}\), there is a nondecreasing function
\[
        F_{d,r,?}:\mathbb R_{\ge0}\longrightarrow\mathbb R_{\ge0},
\]
depending only on \(d,r,?\), such that
\[
        \cu\bigl(A^{\ast_?r}\bigr)\le F_{d,r,?}(\cu(A)).
\]
Here \(A\in\Db(\Gm^d)\) for \(?\in\{!,*\}\), and \(A\in\Perv_{\mathrm{int}}(\Gm^d)\) for
\(?=\mathrm{int}\).  The function is independent of the algebraically closed base field and of \(\ell\).
\end{cor}

\begin{proof}
By \ref{item:C3}, \ref{item:C5}, and Lemma~\ref{lem:standard-geometric-complexity}, for each
\(?\in\{!,*\}\) there is a constant \(a_{d,?}\ge1\), depending only on \(d\) and \(?\), such that
\[
        \cu(A\ast_?B)\le a_{d,?}\cu(A)\cu(B).
\]
The constants are independent of the algebraically closed base field and of \(\ell\).  Induction on \(r\)
gives the result for \(!\) and \(*\).  Internal convolution is a subquotient of
\({}^pH^0(A\ast_!B)\)
\cite[Def.~3.7.1, Def.-Prop.~3.7.2]{GabberLoeser1996}.  Thus \ref{item:C6} gives the same bilinear
estimate, with a constant depending only on \(d\), for \(\ast_{\mathrm{int}}\).  A second induction on
\(r\) proves the internal case.
\end{proof}

\subsection{The Mellin transform}
\label{sec:mellin}

Gabber--Loeser \cite[\S3.3]{GabberLoeser1996} attach to \(\Gm^d\) two Mellin transforms
\(\mathscr M_!\) and \(\mathscr M_*\), with values in the bounded derived category of coherent sheaves on
\(\widehat{\Gm^d}\), together with a natural transformation \(\mathscr M_!\to\mathscr M_*\).

For \(\chi\in\widehat{\Gm^d}(\ol\Ql)\), let \(i_\chi:\operatorname{Spec}\ol\Ql\to\widehat{\Gm^d}\) be the
corresponding point and let \(Li_\chi^*\) denote derived restriction.  For
\(K,A,B\in\Db(\Gm^d)\) and using notation from \S\ref{sec:convolution}, we write
\begin{equation}
\label{eq:comparison-notation}
\begin{gathered}
        \varphi_K\colon \mathscr M_!(K)\longrightarrow\mathscr M_*(K),
        \qquad
        R\Gamma_c(\Gm^d,K_\chi)
        \xrightarrow{\varphi_{K,\chi}}
        R\Gamma(\Gm^d,K_\chi),\\
        \Theta_{A,B,\chi}
        :=\varphi_{A\ast_*B,\chi}\circ R\Gamma_c(c_{A,B,\chi})\colon
        R\Gamma_c\bigl(\Gm^d,(A\ast_!B)_\chi\bigr)
        \longrightarrow R\Gamma\bigl(\Gm^d,(A\ast_*B)_\chi\bigr).
\end{gathered}
\end{equation}

We record the properties of the Mellin transform and convolution used below.
\begin{enumerate}[label=\textup{(GL\arabic*)}, ref=\textup{(GL\arabic*)}]
\item\label{item:GL1} \emph{Fibres} \cite[Cor.~3.3.2]{GabberLoeser1996}.  One has
\(Li_\chi^*\mathscr M_!(K)\simeq
R\Gamma_c(\Gm^d,K_\chi)\) and \(Li_\chi^*\mathscr M_*(K)\simeq R\Gamma(\Gm^d,K_\chi)\).  Under these
identifications, \(Li_\chi^*(\varphi_K)=\varphi_{K,\chi}\)
\cite[Prop.~3.3.4(i)]{GabberLoeser1996}.
\item\label{item:GL2} \emph{Support of the cone} \cite[Def.~3.3.3, Prop.~3.3.4(i)]{GabberLoeser1996}.
\[
        \Delta(K)
        =\Supp\Cone\bigl(\mathscr M_!(K)\xrightarrow{\varphi_K}\mathscr M_*(K)\bigr).
\]
Its \(\ol\Ql\)-points are
\[
        \Delta(K)(\ol\Ql)
        =\bigl\{\chi\in\widehat{\Gm^d}(\ol\Ql):
        R\Gamma_c(\Gm^d,K_\chi)\xrightarrow{\varphi_{K,\chi}}
        R\Gamma(\Gm^d,K_\chi)\text{ is \emph{not} an isomorphism}\bigr\}.
\]
\item\label{item:GL3} \emph{Monoidality} \cite[Prop.~3.3.1(f)]{GabberLoeser1996}.  There are canonical
isomorphisms \(\mathscr M_!(A\ast_!B)\simeq\mathscr M_!(A)\otimes^L\mathscr M_!(B)\) and
\(\mathscr M_*(A\ast_*B)\simeq\mathscr M_*(A)\otimes^L\mathscr M_*(B)\), natural in \(A\) and \(B\).  The
fibrewise compatibility of the comparison \(\mathscr M_!\to\mathscr M_*\) with convolution used below is
recorded in Lemma~\ref{lem:comparison-monoidal}.
\item\label{item:GL4} \emph{Localization}.  The support theorem
\cite[Th.~6.1.1(b)]{GabberLoeser1996} gives
\[
        \Supp\mathscr M_!(K)=\Supp\mathscr M_*(K)
        \qquad(K\in\Db(\Gm^d)).
\]
For closed \(Z\subseteq\widehat{\Gm^d}\), let \(\Db(\Gm^d)_Z\) be the Verdier quotient of
\(\Db(\Gm^d)\) by the full triangulated subcategory of objects whose Mellin transforms are supported in
\(Z\).  Then \(K=0\) in \(\Db(\Gm^d)_Z\) if and only if
\(\Supp\mathscr M_!(K)\subseteq Z\), equivalently if and only if
\(\Supp\mathscr M_*(K)\subseteq Z\).  The perverse \(t\)-structure descends to
\(\Db(\Gm^d)_Z\), and the localization functor is \(t\)-exact
\cite[Prop.~3.6.1, Lem.~3.9.1]{GabberLoeser1996}.
The negligible complexes form a thick triangulated subcategory of \(\Db(\Gm^d)\).
\item\label{item:GL5} \emph{Localization criterion for morphisms}
\cite[Prop.~3.9.2]{GabberLoeser1996}.  Suppose that \(f:A\to B\) is a morphism in \(\Db(\Gm^d)\) such that
the composite
\[
        \mathscr M_!(A)\xrightarrow{\mathscr M_!(f)}\mathscr M_!(B)\longrightarrow\mathscr M_*(B)
\]
is an isomorphism outside a closed subset
\(Z\subseteq\widehat{\Gm^d}\), then \(f\) is an isomorphism in \(\Db(\Gm^d)_Z\), and both \(\Delta(A)\) and
\(\Delta(B)\) are contained in \(Z\).
\item\label{item:GL6} \emph{Character scheme}
\cite[\S\S3.2.2--3.2.3, Props.~A.2.2.2, A.2.2.3]{GabberLoeser1996}.  Each connected component of
\(\widehat{\Gm^d}\) is affine, noetherian, regular, and Jacobson, and its closed points are exactly
its \(\ol\Ql\)-points.
\end{enumerate}

\subsection{The tannakian formalism}
\label{sec:tannakian}

Let \(M_0\in\Perv_{\mathrm{int}}^{\mathrm{ari}}(\Gm^d/k)\) be arithmetically semisimple and let
\(M=M_{0,\bar k}\), which is semisimple in \(\Perv_{\mathrm{int}}(\Gm^d)\)
\cite[Lem.~1.28]{ForeyFresanKowalski2021}.
\begin{enumerate}[label=\textup{(T\arabic*)}, ref=\textup{(T\arabic*)}]
\item\label{item:T1} \emph{Tannakian category}
\cite[Def.~3.7.1--Th.~3.7.5]{GabberLoeser1996}, \cite[Th.~3.16]{ForeyFresanKowalski2021}.
Internal convolution makes
\(\langle M\rangle\), the subcategory of \(\Perv_{\mathrm{int}}(\Gm^d)\) generated by
\(M\) under \(\ast_{\mathrm{int}}\), convolution duals, direct sums, subquotients, and summands, a
neutral tannakian category.  Its unit object is the skyscraper sheaf \(\delta_1\) at the identity
\cite[Def.-Prop.~3.6.5, Prop.~3.7.4(ii)]{GabberLoeser1996}.

Let \(\langle M_0\rangle^{\mathrm{ari}}\) be the full tensor subcategory of
\(\Perv_{\mathrm{int}}^{\mathrm{ari}}(\Gm^d/k)\) tensor-generated by \(M_0\), and write
\(\langle M\rangle^{\mathrm{geo}}=\langle M\rangle\).  Both categories are neutral tannakian.  There are
algebraic groups \(G_M^{\mathrm{ari}}\) and \(G_M^{\mathrm{geo}}\) over \(\ol\Ql\) and tensor equivalences
\[
        \langle M_0\rangle^{\mathrm{ari}}
        \simeq\operatorname{Rep}_{\ol\Ql}\bigl(G_M^{\mathrm{ari}}\bigr),
        \qquad
        \langle M\rangle^{\mathrm{geo}}
        \simeq\operatorname{Rep}_{\ol\Ql}\bigl(G_M^{\mathrm{geo}}\bigr).
\]
These are the arithmetic and geometric tannakian groups.  Geometric base change is a tensor functor and
induces a closed immersion \(G_M^{\mathrm{geo}}\hookrightarrow G_M^{\mathrm{ari}}\)
\cite[Th.~3.30, Def.~3.31, Prop.~3.32]{ForeyFresanKowalski2021}.

\item\label{item:T2} \emph{Semisimplicity and generators}.  Since \(M\) is semisimple, its tannakian
group is reductive and \(\langle M\rangle\) is semisimple \cite[Cor.~3.19]{ForeyFresanKowalski2021}.
The self-dual object \(M\oplus M^\vee\) is a tensor generator, so every object of \(\langle M\rangle\) is
a direct summand of a finite direct sum of its internal convolution powers \cite[Prop.~3.1]{Del82},
\cite[Prop.~1.30]{ForeyFresanKowalski2021}.

\item\label{item:T3} \emph{Arithmetic objects and representations}.  If \(M_0\) is pure of weight
zero, then the category
\(\langle M_0\rangle^{\mathrm{ari}}\) is semisimple.  An irreducible algebraic representation \(\rho\)
of \(G^{\mathrm{ari}}_M\) corresponds to an arithmetically semisimple object \(N_0=\rho(M_0)\) that is
pure of weight zero \cite[Th.~3.30 and Th.~3.33]{ForeyFresanKowalski2021}.
For some \(m(\rho)\ge1\), this object is a direct summand of
\[
        L_0^{\ast_{\mathrm{int}}m(\rho)},
        \qquad
        L_0=M_0\oplus M_0^\vee,
\]
and its geometric base change is a direct summand of
\(L^{\ast_{\mathrm{int}}m(\rho)}\), where \(L=M\oplus M^\vee\)
\cite[Prop.~3.1]{Del82}, \cite[Cor.~3.19]{ForeyFresanKowalski2021}.

\item\label{item:T4} \emph{Cohomology of convolution} \cite[Lem.~1.15]{ForeyFresanKowalski2021}.
For \(A,B\in\Db(\Gm^d)\) and \(\chi\in\widehat{\Gm^d}(\ol\Ql)\), write
\[
\begin{aligned}
 \kappa_!\colon
 R\Gamma_c(\Gm^d,A_\chi)\otimes R\Gamma_c(\Gm^d,B_\chi)
   &\xrightarrow{\ \sim\ } R\Gamma_c\bigl(\Gm^d,(A\ast_!B)_\chi\bigr),\\
 \kappa_*\colon
 R\Gamma(\Gm^d,A_\chi)\otimes R\Gamma(\Gm^d,B_\chi)
   &\xrightarrow{\ \sim\ } R\Gamma\bigl(\Gm^d,(A\ast_*B)_\chi\bigr)
\end{aligned}
\]
for the K\"unneth isomorphisms.  Their compatibility with the forget-supports morphisms is
Lemma~\ref{lem:comparison-monoidal}.

\item\label{item:T5} \emph{Unramified characters and Frobenius}
\cite[Def.~3.25, Lem.~3.28, \S\S3.8--3.9]{ForeyFresanKowalski2021}.  A character
\(\chi\in X_w(M)\), where \(X_w(M)\) is as in \eqref{eq:weak-locus}, is \emph{unramified} for \(M\) if
\(\omega_\chi:N\mapsto H^0(\Gm^d,N_\chi)\) is a fibre functor on \(\langle M\rangle\).  We write
\(X(M)\subseteq X_w(M)\) for the set of unramified characters.  If
\(M_0\in\Perv_{\mathrm{int}}^{\mathrm{ari}}(\Gm^d/k)\) and \(M=M_{0,\bar k}\), define
\(X(M_0)=X(M)\).

For \(L=M\oplus M^\vee\), let \(\alpha_1=\operatorname{id}_L\) and, for \(m\ge2\), put
\[
 \alpha_m
 =c_{L^{\ast_*(m-1)},L}\circ\bigl(\alpha_{m-1}\ast_!\operatorname{id}_L\bigr)
 \colon L^{\ast_!m}\longrightarrow L^{\ast_*m},
 \qquad
 C_m=\Cone(\alpha_m).
\]
Write \(\mu_m\colon(\Gm^d)^m\to\Gm^d\) for the \(m\)-fold multiplication, so that
\(L^{\ast_!m}=R\mu_{m,!}(L^{\boxtimes m})\) and \(L^{\ast_*m}=R\mu_{m,*}(L^{\boxtimes m})\).  Then
\(\alpha_m=\varphi_{\mu_m}(L^{\boxtimes m})\).  This holds for \(m=1\), and for \(m\ge2\) it follows by
induction from Lemma~\ref{lem:forget-supports}(2) applied to
\(\mu_m=m_T\circ(\mu_{m-1}\times\operatorname{id})\).  In that identity the outer factor is
\(\varphi_{m_T}\bigl(L^{\ast_*(m-1)}\boxtimes L\bigr)=c_{L^{\ast_*(m-1)},L}\), and the inner one is
\(\alpha_{m-1}\ast_!\operatorname{id}_L\) by Lemma~\ref{lem:forget-supports-products}(1) together with
\(\varphi_{\operatorname{id}}=\operatorname{id}\) from Lemma~\ref{lem:forget-supports}(1).  The
associativity constraints match the parenthesizations of the iterated convolutions compatibly with
these morphisms, so \(\alpha_m\) is the canonical morphism \(L^{\ast_!m}\to L^{\ast_*m}\).  Every
\(C_m\) is negligible,\footnote{For \(m=2\) one has \(C_2=\Cone(c_{L,L})\), which is negligible by
\S\ref{sec:convolution}.  Suppose that \(C_{m-1}\) is negligible.  The cone of
\(\alpha_{m-1}\ast_!\operatorname{id}_L\) is \(C_{m-1}\ast_!L\), hence is negligible, and so is the
cone of \(c_{L^{\ast_*(m-1)},L}\), again by \S\ref{sec:convolution}.  The negligible complexes form
a thick triangulated subcategory by \ref{item:GL4}, so the octahedral axiom applied to the
definition of \(\alpha_m\) gives the assertion for \(C_m\).} so \(\mathcal N(C_m)\) in the sense of
\S\ref{sec:negligible} is the set defined in \cite[\S3.4]{ForeyFresanKowalski2021}, where
negligibility of the complex is part of the definition.  Hence
\cite[Lem.~3.28]{ForeyFresanKowalski2021} gives
\[
        X_w(M)\cap\bigcap_{m\ge2}\mathcal N(C_m)\subseteq X(M).
\]

Suppose that \(M_0\) is arithmetically semisimple and pure of weight zero.  For \(n\ge1\) and
\(\chi\in X(M_0)(k_n)\), geometric Frobenius \(\operatorname{Fr}_{k_n}\) acts on
\[
        \omega_\chi(N_0)=H^0(\Gm^d,N_\chi),
        \qquad N=N_{0,\bar k},
\]
for every \(N_0\in\langle M_0\rangle^{\mathrm{ari}}\).  Every such \(N_0\) is pure of weight zero, and
\(\omega_\chi(N_0)\) with its Frobenius action is pure of weight zero
\cite[Th.~3.33, \S3.9]{ForeyFresanKowalski2021}.  These actions define a tensor automorphism of
\(\omega_\chi\), hence a conjugacy class
\(\operatorname{Fr}_{M_0,k_n}(\chi)\) in \(G_M^{\mathrm{ari}}(\ol\Ql)\).  Using the isomorphism
\(\iota\) fixed in \S\ref{sec:equidistribution}, fix a maximal compact subgroup
\(K_M\subseteq G_M^{\mathrm{ari}}(\bbC)\).  Following
\cite[\S3.9]{ForeyFresanKowalski2021}, the Frobenius conjugacy class determines a unitary Frobenius
conjugacy class
\[
        \Theta_{M_0,k_n}(\chi)\in K_M^\sharp.
\]
For an algebraic representation \(\rho\) of \(G_M^{\mathrm{ari}}\), put \(N_0=\rho(M_0)\) and
\(N=N_{0,\bar k}\).  By \cite[Lem.~3.35(2)]{ForeyFresanKowalski2021}, applied over \(k_n\), the
character \(\chi\) is unramified for \(N_0\).  Hence
\(\HH^i_c(\Gm^d,N_\chi)=0\) for \(i\ne0\), and the trace identity used below is
\[
 \operatorname{Tr}\rho\bigl(\Theta_{M_0,k_n}(\chi)\bigr)
 =\operatorname{Tr}\bigl(\operatorname{Fr}_{k_n}\mid\HH^0_c(\Gm^d,N_\chi)\bigr).
\]

\end{enumerate}

\begin{lem}
\label{lem:comparison-monoidal}
With the notation of \ref{item:T4}, the following diagram commutes for \(A,B\in\Db(\Gm^d)\) and
\(\chi\in\widehat{\Gm^d}(\ol\Ql)\).
\[
\begin{tikzcd}[column sep=large]
R\Gamma_c(\Gm^d,A_\chi)\otimes R\Gamma_c(\Gm^d,B_\chi)
    \arrow[r,"\varphi_{A,\chi}\otimes\varphi_{B,\chi}"]
    \arrow[d,"\kappa_!"',"\rotatebox{90}{\(\sim\)}"]
& R\Gamma(\Gm^d,A_\chi)\otimes R\Gamma(\Gm^d,B_\chi)
    \arrow[d,"\kappa_*","\rotatebox{90}{\(\sim\)}"'] \\
R\Gamma_c\bigl(\Gm^d,(A\ast_!B)_\chi\bigr)
    \arrow[r,"\Theta_{A,B,\chi}"']
& R\Gamma\bigl(\Gm^d,(A\ast_*B)_\chi\bigr).
\end{tikzcd}
\]
\end{lem}

\begin{proof}
Write \(f\colon \Gm^d\to\operatorname{Spec}\bar k\) for the structure morphism and
\(m\colon \Gm^d\times \Gm^d\to \Gm^d\) for multiplication.  Apply
\eqref{eq:forget-supports-composition-lisse} with \(p=m\), \(K=A\boxtimes B\), and
\(\cL=\cL_\chi\).  Under \(m^*\cL_\chi\simeq\cL_\chi\boxtimes\cL_\chi\), its upper arrow is
\(\varphi_{f\times f}(A_\chi\boxtimes B_\chi)\), and its lower composite is
\(\Theta_{A,B,\chi}\).  The external K\"unneth morphisms followed by \(\beta_!\) and \(\beta_*\) are
the maps \(\kappa_!\) and \(\kappa_*\) of \ref{item:T4}.  Lemma~\ref{lem:forget-supports-products}(1)
therefore gives the stated diagram.
\end{proof}

\section{The relative Mellin transform over a base}
\label{sec:relative-mellin}

Throughout this section, \(S\) is a quasi-projective \(\bar k\)-variety with a fixed locally closed
embedding
\begin{equation}
\label{eq:relative-base-embedding}
        v:S\hookrightarrow\bbP^{N_S}.
\end{equation}
Put \(\bbP^1_S=\bbP^1\times_{\bar k}S\) and
\(\bbG_{m,S}=\Gm\times_{\bar k}S\).  Let
\[
        w_v:\bbG_{m,S}\hookrightarrow\bbP^1\times\bbP^{N_S}\hookrightarrow\bbP^{2N_S+1}
\]
be the immersion induced by \(v\) and the Segre embedding.  Complexities on \(S\) and \(\bbG_{m,S}\)
are denoted by \(c_v\) and \(c_{w_v}\), respectively.

For \(x\in\{0,\infty\}\), with \(U_x\) and \(t_x\) as in
\S\ref{sec:embeddings}, let
\[
\begin{aligned}
 j_x&:\bbG_{m,S}\hookrightarrow U_x\times S,\\
 f_x&=t_x\circ\operatorname{pr}_1:U_x\times S\longrightarrow\bbA^1.
\end{aligned}
\]
Equip \(U_x\times S\) with the locally closed embedding
\[
        v_x:U_x\times S
        \hookrightarrow\bbP^1\times\bbP^{N_S}\hookrightarrow\bbP^{2N_S+1}
\]
induced by \(v\) and the Segre embedding.
The strict henselian traits \(\bbP^1_{(x)}\), their generic points, and the inertia notation are those
of \S\ref{sec:tame-kummer}.  Via \(t_x\), identify \(\bbP^1_{(x)}\) with \(\bbA^1_{(0)}\), and put
\((\bbP^1_{(x)})_S=\bbP^1_{(x)}\times_{\bar k}S\).  Let
\[
 q_x:(\bbP^1_{(x)})_S\longrightarrow\bbP^1_{(x)},
 \qquad
 \rho_x:(\bbP^1_{(x)})_S\longrightarrow U_x\times S
\]
be the projection and the natural base-change morphism.  Thus \(q_x\) identifies with the base change
of \(f_x\) to \(\bbA^1_{(0)}\).  Let
\[
 \iota_x:S\hookrightarrow(\bbP^1_{(x)})_S,
 \qquad
 \bar j_x:\bar\eta_x\times S\longrightarrow(\bbP^1_{(x)})_S
\]
be the special- and geometric-generic-fibre maps, and let
\(\nu_x:\eta_x\times S\to\bbG_{m,S}\) be the natural morphism.  For \(A\in\Db(\bbG_{m,S})\), put
\[
        R\Psi_x A:=R\Psi_{q_x}\bigl(\rho_x^*Rj_{x,*}A\bigr)\in\Db(S),
\]
where \(R\Psi_{q_x}\) is the nearby-cycle functor over the trait \(\bbP^1_{(x)}\)
\cite[Exp.~XIII, \S2.1]{SGA7II}.  The complex \(R\Psi_xA\) carries its continuous \(I_x\)-action, and its
cohomology sheaves are constructible
\cite[\textit{Th\'eor\`emes de finitude}, Th.~3.2]{SGA45}.  Consequently,
Lemma~\ref{lem:eigenvalue-strata} applies to every automorphism induced by an element of \(I_x\).
The diagram
\[
\begin{tikzcd}[column sep=large]
\bar\eta_x\times S \arrow[r,"\bar j_x"] \arrow[d]
&(\bbP^1_{(x)})_S \arrow[d,"\rho_x"]\\
\eta_x\times S \arrow[r] \arrow[d,"\nu_x"']
&U_x\times S\\
\bbG_{m,S}\arrow[ur,"j_x"']
\end{tikzcd}
\]
commutes.  Since \(j_x^*Rj_{x,*}A\simeq A\), it identifies the restriction of
\(\rho_x^*Rj_{x,*}A\) to \(\bar\eta_x\times S\) with
\((\nu_x^*A)|_{\bar\eta_x\times S}\).  The strict-henselian description of nearby cycles
\cite[Exp.~XIII, (2.1.2.3)]{SGA7II} therefore gives an
\(I_x\)-equivariant identification
\begin{equation}
\label{eq:strict-henselian-nearby-cycles}
        R\Psi_x A\simeq
        \iota_x^*R\bar j_{x,*}
        \bigl((\nu_x^*A)|_{\bar\eta_x\times S}\bigr).
\end{equation}

Using the complexity notation of \ref{item:C1}, put
\begin{equation}
\label{eq:relative-geometric-complexity}
 D(S,v):=
 \max\left\{
 c(w_v),\,c(a),\,c(v_x),\,c_{w_v,v_x}(j_x),\,c_{v_x,a}(f_x)
 \ \middle|\ x\in\{0,\infty\}
 \right\}.
\end{equation}
Thus \(D(S,v)\) records the complexities of the embedded varieties and morphisms used to form the
relative nearby cycles.  For \(r\ge1\), Lemma~\ref{lem:standard-geometric-complexity} bounds
\(D(\Gm^r,u_r)\) solely in terms of \(r\).  The same lemma bounds
\(D(\operatorname{Spec}\bar k,\operatorname{id}_{\bbP^0})\) by an absolute constant.  The complexity of
\(R\Psi_x A\) is bounded as follows.

\begin{lem}
\label{lem:coordinate-nearby-complexity}
For every \(N,D\in\mathbb Z_{\ge0}\), there is a constant \(g_{N,D}\ge1\) such that, whenever \(N_S=N\) and
\(D(S,v)\le D\),
\[
        c_v(R\Psi_x A)\le g_{N,D}\,c_{w_v}(A)
        \qquad(A\in\Db(\bbG_{m,S}),\ x=0,\infty).
\]
The constant is independent of \(\bar k\) and of \(\ell\).
\end{lem}

\begin{proof}
With respect to the displayed embeddings, the six-operation estimates \ref{item:C3} bound the
complexity of \(Rj_{x,*}A\) by a constant depending only on \((N,D)\) times \(c_{w_v}(A)\).
The nearby-cycle estimate \ref{item:C7}, applied to the flat morphism \(f_x\) at \(0\in\bbA^1\) with input
\(Rj_{x,*}A\), bounds the complexity of
\(R\Psi_{q_x}\bigl(\rho_x^*Rj_{x,*}A\bigr)=R\Psi_xA\) with respect to the induced embedding of the
special fibre \(\{x\}\times S\).  This embedding equals \(\sigma_x\circ v\), where
\(\sigma_x:\bbP^{N_S}\hookrightarrow\bbP^{2N_S+1}\) is the linear closed immersion obtained from the
Segre embedding by fixing the first factor at \(x\).
Lemma~\ref{lem:linear-reembedding} converts the resulting estimate into the asserted bound for
\(c_v(R\Psi_xA)\).  Taking the larger constant for \(x=0,\infty\) gives \(g_{N,D}\).
\end{proof}

\begin{lem}
\label{lem:inertia-eigenvalue}
Fix \(x\in\{0,\infty\}\), and let \(\widetilde\tau_x\in I_x\) be a lift of the topological generator
\(\tau_x\in I_x^{\mathrm t}\) chosen in \S\ref{sec:tame-kummer}.  Let \(V\) be a finite-dimensional
continuous \(\ol\Ql\)-representation of \(I_x\), and let \(\xi\) be a continuous
tame character of \(I_x\), viewed also as a character of \(I_x^{\mathrm t}\).  If
\[
        H^a_{\mathrm{cts}}(I_x,V\otimes\xi)\neq0
\]
for some \(a\in\mathbb Z\), then the scalar \(\xi(\tau_x)^{-1}\) is an eigenvalue of
\(\widetilde\tau_x\) on \(V\).
\end{lem}

\begin{proof}
The wild inertia group \(P_x\) is pro-\(p\), and \(\xi\) is trivial on \(P_x\).  Apply
Hochschild--Serre to
\(1\to P_x\to I_x\to I_x^{\mathrm t}\to1\).  Since \(\ell\neq p\),
Lemma~\ref{lem:prime-to-ell-vanishing} gives
\[
 R\Gamma_{\mathrm{cts}}(I_x,V\otimes\xi)
 \simeq R\Gamma_{\mathrm{cts}}(I_x^{\mathrm t},V^{P_x}\otimes\xi).
\]
Recall from \S\ref{sec:tame-kummer} the decomposition \(I_x^{\mathrm t}=J_x\times\Gamma_x\) with
\(\tau_x=(\tau_{x,J},\gamma_x)\).  Put \(V_0=(V^{P_x}\otimes\xi)^{J_x}\).  Hochschild--Serre and
Lemma~\ref{lem:prime-to-ell-vanishing}, now applied to
\(1\to J_x\to I_x^{\mathrm t}\to\Gamma_x\to1\), give
\[
 R\Gamma_{\mathrm{cts}}(I_x^{\mathrm t},V^{P_x}\otimes\xi)
 \simeq R\Gamma_{\mathrm{cts}}(\Gamma_x,V_0).
\]
The complex on the right is nonzero.  By Lemma~\ref{lem:zell-cohomology}, the endomorphism
\(\gamma_x-1\) of the finite-dimensional vector space \(V_0\) is not invertible.  Hence there is
a nonzero \(u\in V_0\) fixed by \(\gamma_x\).  Since \(u\) is also fixed by \(J_x\), it is
fixed by \(\tau_x\) on \(V^{P_x}\otimes\xi\).  Write \(u=v\otimes1\), with
\(0\neq v\in V^{P_x}\).  Since \(\xi\) factors through \(I_x^{\mathrm t}\), the fixed-vector equation
for the chosen lift is
\(\xi(\tau_x)\widetilde\tau_x v=v\), and hence
\[
        \widetilde\tau_x v=\xi(\tau_x)^{-1}v.
\]
Thus \(v\) is the required eigenvector.  Since \(P_x\) fixes \(v\), every lift of
\(\tau_x\) acts on \(v\) in the same way.
\end{proof}

Let \(p:\bbG_{m,S}\to S\) and \(q:\bbG_{m,S}\to\Gm\) be the projections.  For
\(A\in\Db(\bbG_{m,S},\ol\Ql)\) and a Kummer character \(\lambda\), define
\[
        B_{!,\lambda}=Rp_!(A\otimes q^*\cL_\lambda),
        \qquad
        B_{*,\lambda}=Rp_*(A\otimes q^*\cL_\lambda).
\]
For \(S=\operatorname{Spec}\bar k\) these are \(B_{!,\lambda}=R\Gamma_c(\Gm,A_\lambda)\) and
\(B_{*,\lambda}=R\Gamma(\Gm,A_\lambda)\).  The transformation \eqref{eq:forget-supports} for \(p\) supplies
the canonical morphism
\(\varphi_p(A\otimes q^*\cL_\lambda)\colon B_{!,\lambda}\to B_{*,\lambda}\).  The exceptional set of
the relative Mellin transform is
\begin{equation}
\label{eq:relative-exceptional-set}
        E_{\Gm}(A):=
        \{\lambda\in\widehat{\Gm}(\ol\Ql):\ B_{!,\lambda}\to B_{*,\lambda}
        \text{ is \emph{not} an isomorphism}\}
\end{equation}
A Kummer character \(\lambda\) is \emph{exceptional} for \(A\) if
\(\lambda\in E_{\Gm}(A)\), and \emph{non-exceptional} otherwise.  For non-exceptional \(\lambda\), write
\(B_\lambda\) for the common value of the canonically isomorphic complexes \(B_{!,\lambda}\) and
\(B_{*,\lambda}\).

For a perverse sheaf on \(\bbG_{m,S}\), the two relative direct images have the following perverse
amplitudes.

\begin{lem}
\label{lem:relative-amplitude}
If \(F\) is a perverse sheaf on \(\bbG_{m,S}\), then
\[
        Rp_!F\in{}^pD^{[0,1]}(S),
        \qquad
        Rp_*F\in{}^pD^{[-1,0]}(S).
\]
\end{lem}

\begin{proof}
The morphism \(p\) is affine, so Artin affine vanishing \cite[Th.~4.1.1, Cor.~4.1.2]{BBDG18},
\cite[Exp.~XIV, 3.1]{SGA4} gives \(Rp_!F\in{}^pD^{\ge0}(S)\) and \(Rp_*F\in{}^pD^{\le0}(S)\).  The
fibres of \(p\) are of dimension one, so the relative-dimension estimate \cite[4.2.4]{BBDG18} gives
\(Rp_!F\in{}^pD^{\le1}(S)\) and \(Rp_*F\in{}^pD^{\ge-1}(S)\).  Combining the two pairs of bounds
gives the claim.
\end{proof}

Let \(\bar p:\bbP^1_S\to S\) be the projection and
\(j:\bbG_{m,S}\hookrightarrow\bbP^1_S\) the open immersion.  For
\(x\in\{0,\infty\}\), let
\(i_x:S\simeq\{x\}\times S\hookrightarrow\bbP^1_S\) be the boundary inclusion.  For a
Kummer character \(\lambda\), put
\(A_\lambda=A\otimes q^*\cL_\lambda\).  Since \(\bar p\) is proper,
\[
        B_{!,\lambda}=R\bar p_*j_!A_\lambda,
        \qquad
        B_{*,\lambda}=R\bar p_*Rj_*A_\lambda ,
\]
and, by Lemma~\ref{lem:forget-supports}(1), the forget-supports morphism \(B_{!,\lambda}\to B_{*,\lambda}\) is
\(R\bar p_*\) applied to \(j_!A_\lambda\to Rj_*A_\lambda\).
The localization triangle for the open immersion \(j\) is
\[
        j_!A_\lambda\longrightarrow Rj_*A_\lambda
        \longrightarrow
        i_{0,*}i_0^*Rj_*A_\lambda\oplus i_{\infty,*}i_\infty^*Rj_*A_\lambda
        \xrightarrow{+1}.
\]
Applying \(R\bar p_*\), and using \(\bar p\circ i_x=\operatorname{id}_S\), gives a canonical
isomorphism in \(\Db(S)\)
\begin{equation}
\label{eq:relative-boundary-cone}
        \Cone(B_{!,\lambda}\to B_{*,\lambda})
        \simeq
        i_0^*Rj_*A_\lambda\oplus i_\infty^*Rj_*A_\lambda .
\end{equation}
For \(x\in\{0,\infty\}\), set
\begin{equation}
\label{eq:boundary-exceptional-set}
        E_x(A):=
        \{\lambda\in\widehat{\Gm}(\ol\Ql):i_x^*Rj_*A_\lambda\neq0\}.
\end{equation}
Equations \eqref{eq:relative-exceptional-set}, \eqref{eq:relative-boundary-cone}, and
\eqref{eq:boundary-exceptional-set} give
\[
        E_{\Gm}(A)=E_0(A)\cup E_\infty(A).
\]

We now bound the exceptional set \eqref{eq:relative-exceptional-set} in terms of \(c_{w_v}(A)\),
\(N_S\), and the parameter \(D(S,v)\) of \eqref{eq:relative-geometric-complexity}.

\begin{lem}
\label{lem:exceptional}
Let \(v:S\hookrightarrow\bbP^{N_S}\) be the fixed embedding
\eqref{eq:relative-base-embedding}, let \(D(S,v)\) be defined by
\eqref{eq:relative-geometric-complexity}, and retain the functions \(H_N\) of
Lemma~\ref{lem:lissity-complexity}.
For every \(N,D\in\mathbb Z_{\ge0}\), there is a nondecreasing function
\(F_{N,D}:\mathbb R_{\ge0}\to\mathbb R_{\ge0}\) such that, if \(N_S=N\) and \(D(S,v)\le D\), then
\[
        \#E_{\Gm}(A)\le F_{N,D}(c_{w_v}(A))
        \qquad(A\in\Db(\bbG_{m,S},\ol\Ql)).
\]
Choosing \(g_{N,D}\) as in Lemma~\ref{lem:coordinate-nearby-complexity}, one may take
\[
        F_{N,D}(t)=2H_N(g_{N,D}t)^2.
\]
The function \(F_{N,D}\) is independent of \(\bar k\) and of \(\ell\).
\end{lem}

\begin{proof}
For each \(x\in\{0,\infty\}\), choose a lift
\(\widetilde\tau_x\in I_x\) of \(\tau_x\).  The \(I_x\)-action on \(R\Psi_x A\) makes
\(\widetilde\tau_x\) an automorphism of this complex.
Fix \(x\) and \(\lambda\in E_x(A)\).  By \eqref{eq:boundary-exceptional-set}, there is a geometric point
\(\bar y\to S\) such that \((i_x^*Rj_*A_\lambda)_{\bar y}\neq0\).  The boundary stalk is computed by the
Hochschild--Serre spectral sequence for nearby cycles over the strictly henselian trait
\[
        E_2^{a,b}
        =H^a_{\mathrm{cts}}\bigl(I_x,H^b\bigl((R\Psi_x A_\lambda)_{\bar y}\bigr)\bigr)
        \Longrightarrow
        H^{a+b}\bigl((i_x^*Rj_*A_\lambda)_{\bar y}\bigr),
\]
in which \(E_2^{a,b}=0\) for \(a\notin\{0,1\}\), see \cite[\textit{Th\'eor\`emes de finitude},
Lem.~3.11]{SGA45} and \cite[Th.~(6.5)]{KatzLaumon1985}.  The abutment is nonzero, so there are
integers \(a\) and \(b\) with
\[
        H^a_{\mathrm{cts}}\bigl(I_x,H^b\bigl((R\Psi_x A_\lambda)_{\bar y}\bigr)\bigr)\neq0.
\]
The composite \(q\circ\nu_x:\eta_x\times S\to\Gm\) factors through the projection to \(\eta_x\), so the
restriction of \(\nu_x^*q^*\cL_\lambda\) to \(\bar\eta_x\times S\) is the constant sheaf attached to the
one-dimensional \(\ol\Ql\)-vector space \(V_\lambda=(\cL_\lambda)_{\bar\eta_x}\), on which \(I_x\) acts
through \(I_x\twoheadrightarrow I_x^{\mathrm t}\) by the character \(\lambda_x\) of
\S\ref{sec:tame-kummer}.  Applying \eqref{eq:strict-henselian-nearby-cycles} to \(A_\lambda\) gives
\[
        R\Psi_x A_\lambda\simeq
        \iota_x^*R\bar j_{x,*}
        \bigl((\nu_x^*A)|_{\bar\eta_x\times S}\otimes_{\ol\Ql}V_\lambda\bigr).
\]
Lemma~\ref{lem:constant-tensor}, applied to \(f=\bar j_x\), followed by the canonical tensor
compatibility of \(\iota_x^*\), identifies the right-hand side with
\(R\Psi_x A\otimes_{\ol\Ql}V_\lambda\).  Naturality of these isomorphisms in the complex and in the
vector space, applied respectively to the action of \(\sigma\in I_x\) on
\((\nu_x^*A)|_{\bar\eta_x\times S}\) and to the action of \(\sigma\) on \(V_\lambda\), shows that this
identification is \(I_x\)-equivariant for the diagonal actions.
Hence
\[
        R\Psi_x A_\lambda\simeq R\Psi_x A\otimes_{\ol\Ql}V_\lambda
\]
in \(\Db(S)\), equivariantly for the diagonal \(I_x\)-action on the right.  Since \(V_\lambda\) is
one-dimensional, a choice of basis identifies the underlying complexes and multiplies the
\(I_x\)-action by \(\lambda_x\).  Stalks at \(\bar y\) and cohomology commute with
\(\otimes_{\ol\Ql}V_\lambda\), so \(H^b\bigl((R\Psi_x A_\lambda)_{\bar y}\bigr)\) is the twist
\(H^b\bigl((R\Psi_x A)_{\bar y}\bigr)\otimes\lambda_x\) in the sense of
Lemma~\ref{lem:inertia-eigenvalue}, and
\[
        H^a_{\mathrm{cts}}\bigl(I_x,H^b\bigl((R\Psi_x A)_{\bar y}\bigr)\otimes\lambda_x\bigr)\neq0.
\]
Lemma~\ref{lem:inertia-eigenvalue}, applied to \(V=H^b\bigl((R\Psi_x A)_{\bar y}\bigr)\) and
\(\xi=\lambda_x\), shows that the scalar
\(\lambda_x(\tau_x)^{-1}\) is an eigenvalue of \(\widetilde\tau_x\) on
\(H^b\bigl((R\Psi_x A)_{\bar y}\bigr)\).  By Lemma~\ref{lem:kummer-local-monodromy},
\[
        \#E_x(A)\le
        \#\Eig_S(R\Psi_x A,\widetilde\tau_x).
\]
Lemmas~\ref{lem:eigenvalue-strata} and~\ref{lem:coordinate-nearby-complexity} now give
\[
 \#E_{\Gm}(A)
 \le 2H_N\bigl(g_{N,D}c_{w_v}(A)\bigr)^2
 =F_{N,D}(c_{w_v}(A)).
\]
This is the asserted bound, and it is nondecreasing because \(H_N\) is nondecreasing.
\end{proof}

\begin{cor}
\label{cor:d-one}
There is a nondecreasing function
\(\Phi^\Delta_{1}:\mathbb R_{\ge0}\to\mathbb R_{\ge0}\) such that, for every \(c\ge0\) and
every \(B\in\Db((\Gm)_{\bar k},\ol\Ql)\) with \(\cu(B)\le c\), the set \(\Delta(B)(\ol\Ql)\) is finite and
\[
        \#\Delta(B)(\ol\Ql)\le\Phi^\Delta_{1}(c).
\]
The function is independent of \(\bar k\) and of \(\ell\).
\end{cor}

\begin{proof}
Take \(S=\operatorname{Spec}\bar k\), with its standard embedding.  The constituents of the parameter
\(D(S,v)\) of \eqref{eq:relative-geometric-complexity} are standard embeddings and morphisms of
\S\ref{sec:embeddings} in dimension at most one, so Lemma~\ref{lem:standard-geometric-complexity} gives
\(D(S,v)\le D_1^{\mathrm{std}}\).  Then \(E_{\Gm}(B)=\Delta(B)(\ol\Ql)\) by \ref{item:GL2}, and
Lemma~\ref{lem:exceptional} gives the bound with \(\Phi^\Delta_{1}=F_{0,D_1^{\mathrm{std}}}\).
\end{proof}

For perverse \(B\), this is the characteristic-uniform form of Katz's one-dimensional exceptional-character
bound \cite[Ch.~1]{Katz2012}.

\section{Proofs of Theorems~\ref{thm:delta-count} and~\ref{thm:stratified}}
\label{sec:counting}

We first prove Theorem~\ref{thm:delta-count} for perverse sheaves, together with
Theorem~\ref{thm:stratified}, by induction on \(d\), with the case \(d=1\) supplied by
Corollary~\ref{cor:d-one}.  The case of an arbitrary complex then follows from its perverse cohomology
sheaves.  The inductive step applies the relative Mellin transform of \S\ref{sec:relative-mellin} in the
first coordinate.  For \(d\ge2\), we take the base
\(\Gm^{d-1}\), endowed with the embedding \(u_{d-1}\) of \S\ref{sec:embeddings}, and write
\(\Gm^d=\Gm\times\Gm^{d-1}=\bbG_{m,\Gm^{d-1}}\).
The resulting embedding using the recipe in \S\ref{sec:relative-mellin} is
\[
 w_{u_{d-1}}:\Gm^d\hookrightarrow\bbP^1\times\bbP^{2^{d-1}-1}
 \hookrightarrow\bbP^{2^d-1}.
\]
By the compatibility of the embeddings fixed in \S\ref{sec:embeddings}, this is \(u_d\).
Consequently, \(c_{w_{u_{d-1}}}(A)=c_{u_d}(A)=\cu(A)\) for every \(A\in\Db(\Gm^d)\).

Let
\begin{equation}
\label{eq:coordinate-projections}
 p:\Gm^d\longrightarrow\Gm^{d-1},\qquad
 q:\Gm^d\longrightarrow\Gm
\end{equation}
be the projections to the last \(d-1\) factors and to the first factor, respectively.  Under the
identification \(\Gm^d=\bbG_{m,\Gm^{d-1}}\), these are the structural projections of
\S\ref{sec:relative-mellin}.  For \(A\in\Db(\Gm^d)\) and a Kummer character \(\lambda\), the complexes
\(B_{!,\lambda}\) and \(B_{*,\lambda}\), the exceptional set \(E_{\Gm}(A)\) of
\eqref{eq:relative-exceptional-set}, and, for non-exceptional \(\lambda\), the common value
\(B_\lambda\) are taken with respect to this base.
\begin{lem}
\label{lem:complexity-bounds}
For every \(d\ge1\), there is a nondecreasing function
\(F_{d}:\mathbb R_{\ge0}\to\mathbb R_{\ge0}\), independent of the algebraically closed base field and of
\(\ell\), with the following properties.
\begin{enumerate}
\item[(i)] for \(M\in\Db(\Gm^d)\), with \(\cH^i(M)_1\) the stalks at the identity,
\[
        \sum_i\cu({}^pH^i(M))\le F_{d}(\cu(M)),\qquad
        \sum_i\dim\cH^i(M)_1\le F_{d}(\cu(M)).
\]
\item[(ii)] for semisimple \(M\in\Perv_{\mathrm{int}}(\Gm^d)\),
\[
        \cu(M\oplus M^\vee)\le F_{d}(\cu(M)).
\]
\item[(iii)] if \(d\ge2\), then for every perverse sheaf \(A\) on \(\Gm^d\) and every Kummer character
\(\lambda\) of the first factor, using the notation in (\ref{eq:relative-exceptional-set})
\[
\begin{gathered}
        \#E_{\Gm}(A)\le F_{d}(\cu(A)),\\
        \cu(B_{!,\lambda}),\ \cu(B_{*,\lambda}),\
        \cu({}^pH^r(B_{!,\lambda})),\ \cu({}^pH^r(B_{*,\lambda}))
        \le F_{d}(\cu(A))
        \qquad(r\in\mathbb Z).
\end{gathered}
\]
\end{enumerate}
\end{lem}

\begin{proof}
For (i), let \((A_{i,j})_{1\le j\le n_i}\) be the Jordan--H\"older factors of \({}^pH^i(M)\),
repeated with multiplicity.  The Jordan--H\"older estimate \ref{item:C6} gives
\[
        \sum_i\sum_{j=1}^{n_i}\cu(A_{i,j})
        \ll_d c(u_d)\cu(M).
\]
For each \(i\), repeated application of the triangle inequality \ref{item:C4} to a Jordan--H\"older
filtration gives
\[
        \cu({}^pH^i(M))\le\sum_{j=1}^{n_i}\cu(A_{i,j}).
\]
Lemma~\ref{lem:standard-geometric-complexity} bounds \(c(u_d)\) solely in terms of \(d\), uniformly in
the algebraically closed field and in \(\ell\).  Summing the last inequality over \(i\) proves the first
bound in (i).

Let \(i_1:\{1\}\hookrightarrow\Gm^d\) be the identity section.  Since the complexity of a complex on
the point is the sum of the dimensions of its cohomology groups \ref{item:C1}, one has
\[
        c(i_1^*M)=\sum_i\dim\cH^i(M)_1.
\]
The pullback estimate \ref{item:C3}, together with the uniform bound for the embedded
morphism \(i_1\) in Lemma~\ref{lem:standard-geometric-complexity}, proves the second bound in (i).

For (ii), the duality and pullback estimates \ref{item:C3} bound
\(\cu(M^\vee)\) in terms of \(\cu(M)\), since \(M^\vee=\operatorname{inv}^*D(M)\).  The direct-sum
identity \ref{item:C4} then bounds \(\cu(M\oplus M^\vee)\).

For (iii), Lemma~\ref{lem:kummer-complexity} and the pullback estimate give a bound for
\(\cu(q^*\cL_\lambda)\) independent of \(\lambda\).  The tensor and pushforward estimates
\ref{item:C3} then bound
\[
        B_{!,\lambda}=Rp_!(A\otimes q^*\cL_\lambda),\qquad
        B_{*,\lambda}=Rp_*(A\otimes q^*\cL_\lambda)
\]
and \ref{item:C4} and \ref{item:C6} bound their perverse cohomology sheaves, uniformly in
\(\lambda\).  Lemma~\ref{lem:exceptional}, with \((N,D)=(2^{d-1}-1,D_d^{\mathrm{std}})\), bounds
\(\#E_{\Gm}(A)\).  Take the pointwise maximum of these finitely many nondecreasing functions.  Every input
complexity has now been bounded by Lemma~\ref{lem:standard-geometric-complexity}, so \(F_{d}\) is
independent of the field, of \(\ell\), and of \(\lambda\).
\end{proof}

For each \(d\), fix a function \(F_d\) as in Lemma~\ref{lem:complexity-bounds}.

\begin{lem}
\label{lem:fibre}
Let \(d\ge2\) and let \(A\) be perverse on \(\Gm^d\).  With the notation of
\eqref{eq:relative-exceptional-set}, suppose that \(\lambda\notin E_{\Gm}(A)\) and write \(B_\lambda\)
for the common value of the canonically isomorphic complexes \(B_{!,\lambda}\) and \(B_{*,\lambda}\) on
\(\Gm^{d-1}\).
Then \(B_\lambda\) is perverse and
\[
        \cu(B_\lambda)\le F_{d}(\cu(A)).
\]
One has the following equality of subsets of \(\widehat{\Gm^{d-1}}(\ol\Ql)\):
\[
        \{\eta\in\widehat{\Gm^{d-1}}(\ol\Ql):(\lambda,\eta)\in\Delta(A)(\ol\Ql)\}
        =\Delta(B_\lambda)(\ol\Ql).
\]
\end{lem}

\begin{proof}
Put \(K_\lambda=A\otimes q^*\cL_\lambda\).  The sheaf \(q^*\cL_\lambda\) is lisse of rank one, so
\(K_\lambda\) is perverse.  Since \(p\) is affine, Artin affine vanishing gives
\[
        B_{!,\lambda}=Rp_!K_\lambda\in{}^pD^{\ge0}(\Gm^{d-1}),
        \qquad
        B_{*,\lambda}=Rp_*K_\lambda\in{}^pD^{\le0}(\Gm^{d-1})
\]
\cite[Th.~4.1.1, Cor.~4.1.2]{BBDG18}, \cite[Exp.~XIV, 3.1]{SGA4}.  By the definition of
\(E_{\Gm}(A)\) in \eqref{eq:relative-exceptional-set}, the morphism
\(\varphi_p(K_\lambda)\colon B_{!,\lambda}\to B_{*,\lambda}\) is an isomorphism.  Hence their common value
\(B_\lambda\) belongs to the heart of the perverse \(t\)-structure.  Its complexity bound follows from
Lemma~\ref{lem:complexity-bounds}(iii).

Let \(f\colon\Gm^{d-1}\to\operatorname{Spec}\bar k\) be the structure morphism.  Apply
\eqref{eq:forget-supports-composition-lisse} with \(K=K_\lambda\) and \(\cL=\cL_\eta\).  Identifying
\(Rp_!K_\lambda\) and \(Rp_*K_\lambda\) with \(B_\lambda\) by \(\varphi_p(K_\lambda)\) gives the
commutative square
\[
\begin{tikzcd}[column sep=large]
R\Gamma_c\bigl(\Gm^d,A_{(\lambda,\eta)}\bigr)
  \arrow[r,"\varphi_{A,(\lambda,\eta)}"]
  \arrow[d,"\beta_!"',"\rotatebox{90}{\(\sim\)}"] &
R\Gamma\bigl(\Gm^d,A_{(\lambda,\eta)}\bigr)
  \arrow[d,"\beta_*","\rotatebox{90}{\(\sim\)}"'] \\
R\Gamma_c\bigl(\Gm^{d-1},(B_\lambda)_\eta\bigr)
  \arrow[r,"\varphi_{B_\lambda,\eta}"'] &
R\Gamma\bigl(\Gm^{d-1},(B_\lambda)_\eta\bigr).
\end{tikzcd}
\]
Here \(\beta_!\) and \(\beta_*\) are the composition and projection-formula isomorphisms in
\eqref{eq:forget-supports-composition-lisse}.  Hence the upper horizontal morphism is an isomorphism if
and only if the lower horizontal morphism is an isomorphism.  The equality follows from
\ref{item:GL2}.
\end{proof}

\begin{proof}[Proof of Theorem~\ref{thm:delta-count}]
\emph{Perverse sheaves.}
We construct nondecreasing functions \(\Phi^{\mathrm p}_{d}\) satisfying the asserted estimate for
perverse sheaves.  For \(d=1\), take \(\Phi^{\mathrm p}_{1}=\Phi^\Delta_{1}\) from
Corollary~\ref{cor:d-one}.

Assume \(d\ge2\) and the perverse case known for \(d-1\).  Let \(A\) be perverse on \(\Gm^d\) and put
\(c=\cu(A)\).
For each \(\lambda\in\widehat{\Gm(k_n)}\setminus E_{\Gm}(A)(k_n)\), Lemma~\ref{lem:fibre} and the inductive
hypothesis give
\[
\begin{aligned}
 \#\{\eta\in\widehat{\Gm^{d-1}(k_n)}:(\lambda,\eta)\in\Delta(A)(\ol\Ql)\}
 &=|\Delta(B_\lambda)(k_n)|\\
 &\le\Phi^{\mathrm p}_{d-1}(\cu(B_\lambda))\,q^{n(d-2)}\\
 &\le\Phi^{\mathrm p}_{d-1}
       \bigl(F_{d}(c)\bigr)q^{n(d-2)}.
\end{aligned}
\]
There are at most \(|\widehat{\Gm(k_n)}|\le q^n\) such \(\lambda\).  For the exceptional characters, the bounds
\(|\widehat{\Gm^{d-1}(k_n)}|\le q^{n(d-1)}\) and Lemma~\ref{lem:complexity-bounds}(iii) imply that
\[
        \#E_{\Gm}(A)(k_n)\,|\widehat{\Gm^{d-1}(k_n)}|
        \le
        F_{d}(c)q^{n(d-1)}.
\]
Consequently
\[
        |\Delta(A)(k_n)|
        \le
        \Bigl(\Phi^{\mathrm p}_{d-1}\bigl(F_{d}(c)\bigr)
        +F_{d}(c)\Bigr)q^{n(d-1)}.
\]
Define \(\Phi^{\mathrm p}_{d}\) to be the function in parentheses.  Since \(F_{d}\) and
\(\Phi^{\mathrm p}_{d-1}\) are nondecreasing, so is \(\Phi^{\mathrm p}_{d}\).  This proves the
asserted estimate for perverse sheaves.

\emph{Arbitrary complexes.}
The case \(d=1\) is Corollary~\ref{cor:d-one}.  Assume \(d\ge2\) and let
\(M\in\Db(\Gm^d,\ol\Ql)\).  Induction on the perverse amplitude, using the octahedral axiom, gives a
filtration of \(\Cone(\mathscr M_!(M)\to\mathscr M_*(M))\) whose graded pieces are
\[
        \Cone\bigl(\mathscr M_!({}^{p}H^{i}(M))\to\mathscr M_*({}^{p}H^{i}(M))\bigr)[-i],\qquad i\in\mathbb Z .
\]
The support of an iterated extension is contained in the union of the supports of its graded pieces, and a
shift preserves support, so \ref{item:GL2} gives
\[
        \Delta(M)\subseteq\bigcup_{i}\Delta({}^{p}H^{i}(M)).
\]
The perverse case gives
\[
        |\Delta({}^{p}H^{i}(M))(k_n)|
        \le\Phi^{\mathrm p}_{d}(\cu({}^{p}H^{i}(M)))\,q^{n(d-1)} .
\]
Lemma~\ref{lem:complexity-bounds}(i) gives
\[
        \sum_i\cu({}^{p}H^{i}(M))\le F_{d}(\cu(M)).
\]
Since a nonzero object has complexity at least \(1\) \ref{item:C2}, the \({}^{p}H^{i}(M)\)
are nonzero for at most \(F_{d}(\cu(M))\) values of \(i\), and each has complexity at most
\(F_{d}(\cu(M))\).  We may therefore take
\[
        \Phi^\Delta_{d}(C)
        =F_{d}(C)\,
        \Phi^{\mathrm p}_{d}\bigl(F_{d}(C)\bigr).
\]
This function is nondecreasing.
\end{proof}

\begin{proof}[Proof of Theorem~\ref{thm:stratified}]
We use the notation of \eqref{eq:cohomology-jump-loci}.

We argue by induction on \(d\).  For \(i=0\),
\(|\mathcal V_{d,0}(A)(k_n)|\le|\widehat{\Gm^d(k_n)}|\le q^{nd}\), so we may assume \(1\le i\le d\).  We use
repeatedly that, for every \(d'\ge1\), a perverse sheaf \(P\) on \(\Gm^{d'}\), and a Kummer character
\(\eta\), Artin affine vanishing and its Verdier dual give \(\HH^a_c(\Gm^{d'},P_\eta)=0\) for \(a<0\)
and \(\HH^a(\Gm^{d'},P_\eta)=0\) for \(a>0\) \cite[Th.~4.1.1, Cor.~4.1.2]{BBDG18},
\cite[Exp.~XIV, 3.1]{SGA4}.  The cohomological amplitude of a perverse sheaf on a \(d'\)-dimensional
variety gives \(\HH^a_c(\Gm^{d'},P_\eta)=\HH^a(\Gm^{d'},P_\eta)=0\) for \(|a|>d'\).

\emph{Base case \(d=1\).}  For \(i=1\),
\(\mathcal V_{1,1}(A)\subseteq\Delta(A)(\ol\Ql)\) by Artin vanishing on \(\Gm\), so
Corollary~\ref{cor:d-one} gives
\(|\mathcal V_{1,1}(A)(k_n)|\le\Phi^\Delta_{1}(\cu(A))\).
Together with the constant \(1\) covering the case \(i=0\), we may take \(\Psi_{1}=\Phi^\Delta_{1}+1\).

\emph{Inductive step.}  Let \(d\ge2\), assume the theorem for \(d-1\), and use the projections \(p\) and
\(q\) of \eqref{eq:coordinate-projections}.  For a Kummer character \(\lambda\), the complex
\(A\otimes q^*\cL_\lambda\) is a
perverse sheaf, and \(B_{!,\lambda}\), \(B_{*,\lambda}\), \(E_{\Gm}(A)\) are as in
\eqref{eq:relative-exceptional-set}.  Put \(c=\cu(A)\) and \(E=E_{\Gm}(A)\).
Lemma~\ref{lem:complexity-bounds}(iii) gives \(\#E\le F_{d}(c)\).  The projection formula and
Lemmas~\ref{lem:forget-supports}(2) and~\ref{lem:forget-supports-products}(2) give, for all Kummer
characters \(\lambda,\eta\) and every \(j\),
\[
        \HH^j_c(\Gm^d,A_{(\lambda,\eta)})
        \simeq\HH^j_c(\Gm^{d-1},B_{!,\lambda}\otimes\cL_\eta),
        \qquad
        \HH^j(\Gm^d,A_{(\lambda,\eta)})
        \simeq\HH^j(\Gm^{d-1},B_{*,\lambda}\otimes\cL_\eta).
\]

\emph{Non-exceptional \(\lambda\).}  If \(\lambda\notin E\), then
\(B_{!,\lambda}\simeq B_{*,\lambda}=B_\lambda\) is perverse and satisfies
\(\cu(B_\lambda)\le F_{d}(c)\) by Lemma~\ref{lem:fibre}.  The identifications give
\[
 \{\eta\in\widehat{\Gm^{d-1}}(\ol\Ql):(\lambda,\eta)\in\mathcal V_{d,i}(A)\}
 =\mathcal V_{d-1,i}(B_\lambda)
\]
for \(0\le i\le d-1\).  The set on the left is empty for \(i=d\).  For \(1\le i\le d-1\), the inductive
hypothesis and \(|\widehat{\Gm(k_n)}|\le q^n\) give
\[
        \Bigl|\mathcal V_{d,i}(A)(k_n)\cap
        \bigl((\widehat{\Gm(k_n)}\setminus E(k_n))\times\widehat{\Gm^{d-1}(k_n)}\bigr)\Bigr|
        \le \Psi_{d-1}\bigl(F_{d}(c)\bigr)q^{n(d-i)}.
\]
For \(i=d\), the left-hand side is zero.

\emph{Exceptional \(\lambda\).}  The sheaf \(A\otimes q^*\cL_\lambda\) is perverse, so
Lemma~\ref{lem:relative-amplitude} gives \(B_{!,\lambda}\in{}^pD^{[0,1]}(\Gm^{d-1})\) and
\(B_{*,\lambda}\in{}^pD^{[-1,0]}(\Gm^{d-1})\), with
\(\cu(B_{!,\lambda}),\cu(B_{*,\lambda})\le F_{d}(c)\) by Lemma~\ref{lem:complexity-bounds}(iii).  Write
\(P^r={}^pH^r(B_{!,\lambda})\) for \(r=0,1\).  These sheaves are perverse on \(\Gm^{d-1}\), with
\(\cu(P^r)\le F_{d}(c)\) by the same lemma.  For
\(\eta\in\widehat{\Gm^{d-1}}(\ol\Ql)\), the perverse spectral sequence
\[
 E_2^{a,r}=\HH^a_c(\Gm^{d-1},(P^r)_\eta)
 \Rightarrow\HH^{a+r}_c(\Gm^{d-1},(B_{!,\lambda})_\eta)
\]
and Artin vanishing for \(P^0\) and \(P^1\) give \(\HH^{a}_c(\Gm^{d-1},(B_{!,\lambda})_\eta)=0\) for
\(a<0\), together with the inclusion
\[
\begin{aligned}
 &\{\eta\in\widehat{\Gm^{d-1}}(\ol\Ql):
       \HH^{i}_c(\Gm^{d-1},(B_{!,\lambda})_\eta)\neq0\}\\
 &\qquad\subseteq
 \{\eta\in\widehat{\Gm^{d-1}}(\ol\Ql):\HH^{i}_c(\Gm^{d-1},(P^0)_\eta)\neq0\}\\
 &\qquad{}\cup
 \{\eta\in\widehat{\Gm^{d-1}}(\ol\Ql):\HH^{i-1}_c(\Gm^{d-1},(P^1)_\eta)\neq0\}.
\end{aligned}
\]
The first set on the right lies in
\(\mathcal V_{d-1,i}(P^0)\), and the second lies in \(\mathcal V_{d-1,i-1}(P^1)\) if \(i>1\).  For
\(i=1\), the second is contained in \(\widehat{\Gm^{d-1}(k_n)}\).  The inductive hypothesis therefore gives
\[
 \#\Bigl\{\eta\in\widehat{\Gm^{d-1}(k_n)}:
 \substack{\HH^i_c(\Gm^{d-1},(B_{!,\lambda})_\eta)\neq0\ \,\text{or}\\
           \HH^{-i}_c(\Gm^{d-1},(B_{!,\lambda})_\eta)\neq0}\Bigr\}
 \le\Bigl(2\Psi_{d-1}(F_{d}(c))+1\Bigr)q^{n(d-i)}.
\]

For ordinary cohomology, put \(Q^r={}^pH^r(B_{*,\lambda})\) for \(r=-1,0\).  These sheaves have complexity
at most \(F_{d}(c)\).  The analogous spectral sequence and Artin vanishing give
\(\HH^{a}(\Gm^{d-1},(B_{*,\lambda})_\eta)=0\) for \(a>0\), together with the inclusion
\[
\begin{aligned}
 &\{\eta\in\widehat{\Gm^{d-1}}(\ol\Ql):
       \HH^{-i}(\Gm^{d-1},(B_{*,\lambda})_\eta)\neq0\}\\
 &\qquad\subseteq
 \{\eta\in\widehat{\Gm^{d-1}}(\ol\Ql):\HH^{-i}(\Gm^{d-1},(Q^0)_\eta)\neq0\}\\
 &\qquad{}\cup
 \{\eta\in\widehat{\Gm^{d-1}}(\ol\Ql):\HH^{-i+1}(\Gm^{d-1},(Q^{-1})_\eta)\neq0\}.
\end{aligned}
\]
The same inductive estimates apply to \(Q^0\) and \(Q^{-1}\).  Hence
\[
 \#\Bigl\{\eta\in\widehat{\Gm^{d-1}(k_n)}:
 \substack{\HH^i(\Gm^{d-1},(B_{*,\lambda})_\eta)\neq0\ \,\text{or}\\
           \HH^{-i}(\Gm^{d-1},(B_{*,\lambda})_\eta)\neq0}\Bigr\}
 \le\Bigl(2\Psi_{d-1}(F_{d}(c))+1\Bigr)q^{n(d-i)}.
\]
Multiplying the sum of the two bounds above for \(B_{!,\lambda}\) and \(B_{*,\lambda}\) by
\(\#E\le F_{d}(c)\), the exceptional contribution is at most
\[
 2F_{d}(c)
 \Bigl(2\Psi_{d-1}(F_{d}(c))+1\Bigr)q^{n(d-i)}.
\]

Combining the non-exceptional and exceptional estimates, we may take
\[
\begin{aligned}
        \Psi_{d}(c)
        ={}&1+\Psi_{d-1}\bigl(F_{d}(c)\bigr)\\
        &+2F_{d}(c)
        \Bigl(2\Psi_{d-1}\bigl(F_{d}(c)\bigr)+1\Bigr).
\end{aligned}
\]
The term \(1\) gives the estimate for \(i=0\).  Since \(F_{d}\) and \(\Psi_{d-1}\) are nondecreasing,
so is \(\Psi_{d}\).  This completes the induction.
\end{proof}

\begin{rem}
Theorem~\ref{thm:stratified} should be compared with the recent uniform variants of the Katz--Laumon and
Fouvry--Katz stratification theorems proved by Bonolis, Kowalski, and Woo
\cite{BonolisKowalskiWoo2026} for exponential sums in families.  Those results concern stratifications
of the additive Fourier parameter space.  Here the parameter space is the Gabber--Loeser character
scheme and the estimate counts its finite-field Kummer characters.  The estimate is uniform over
perverse sheaves of bounded complexity while the characteristic and \(\ell\) vary.  This is the form used in
\S\ref{sec:unramified} to prove Corollary~\ref{cor:equidistribution}.
\end{rem}

\section{Proofs of Theorem~\ref{thm:unram-count} and Corollary~\ref{cor:equidistribution}}
\label{sec:unramified}

Let \(M_0\in\Perv_{\mathrm{int}}^{\mathrm{ari}}(\Gm^d/k)\) be arithmetically semisimple, and let
\(M=M_{0,\bar k}\).  By \S\ref{sec:tannakian}, \(M\) is semisimple in
\(\Perv_{\mathrm{int}}(\Gm^d)\).  We use the unramified locus \(X(M_0)\) of \ref{item:T5} in
\S\ref{sec:tannakian} and let \(L=M\oplus M^\vee\).  Then
\(L\) is perverse, self-dual, and lies in \(\Perv_{\mathrm{int}}(\Gm^d)\), with
\(\cu(L)\le F_{d}(\cu(M))\) by Lemma~\ref{lem:complexity-bounds}(ii).

Let \(\alpha_m\) and \(C_m\) be as in \ref{item:T5}.  The twist of \(\alpha_m\) by \(\cL_\chi\) is
denoted \(\alpha_{m,\chi}\).

\begin{lem}
\label{lem:weakly-unramified-inclusion}
Let \(M_0\in\Perv_{\mathrm{int}}^{\mathrm{ari}}(\Gm^d/k)\) be arithmetically semisimple, let
\(M=M_{0,\bar k}\), and let \(L=M\oplus M^\vee\).  Then \(X_w(L)\subseteq X(M_0)\).
\end{lem}

\begin{proof}
With the Mellin transforms of \S\ref{sec:mellin}, write
\(\gamma_m\colon\mathscr M_!(L^{\ast_!m})\to\mathscr M_*(L^{\ast_*m})\) for the morphism induced by
\(\alpha_m\) and the transformation \(\mathscr M_!\to\mathscr M_*\).

\emph{Fibres of \(\gamma_m\).}  For \(\chi\in\widehat{\Gm^d}(\ol\Ql)\), let \(\Theta_{m,\chi}=Li_\chi^*\gamma_m\),
viewed through \ref{item:GL1} as
\[
        R\Gamma_c\bigl(\Gm^d,(L^{\ast_!m})_\chi\bigr)
        \xrightarrow{\ R\Gamma_c(\alpha_{m,\chi})\ }R\Gamma_c\bigl(\Gm^d,(L^{\ast_*m})_\chi\bigr)
        \longrightarrow R\Gamma\bigl(\Gm^d,(L^{\ast_*m})_\chi\bigr),
\]
where the second morphism forgets supports.  Put \(\Theta_{1,\chi}=\varphi_{L,\chi}\).  Naturality of the
K\"unneth isomorphisms and Lemma~\ref{lem:comparison-monoidal} give
\[
        \Theta_{m,\chi}=\Theta_{m-1,\chi}\otimes\varphi_{L,\chi}.
\]
Thus \(\Theta_{m,\chi}=\varphi_{L,\chi}^{\otimes m}\) under the iterated K\"unneth identifications.

\emph{Support of \(\Cone(\gamma_m)\).}  The complex \(\Cone(\gamma_m)\) is coherent because the Mellin
transforms take values in coherent complexes (\S\ref{sec:mellin}). Since \(Li_\chi^*\) is triangulated, one
has \(Li_\chi^*\Cone(\gamma_m)\simeq\Cone(\Theta_{m,\chi})\).  For \(\chi\notin\Delta(L)(\ol\Ql)\), the
morphism \(\varphi_{L,\chi}\), hence \(\Theta_{m,\chi}\), is an isomorphism, so the derived fibre of
\(\Cone(\gamma_m)\) vanishes at every character outside \(\Delta(L)\).  Put
\(U=\widehat{\Gm^d}\setminus\Delta(L)\) and \(Q_m=\Cone(\gamma_m)|_U\).

By \ref{item:GL6}, each connected component of \(\widehat{\Gm^d}\) is noetherian and Jacobson, and its
closed points are its \(\ol\Ql\)-points.  The components are open, and both properties are local, so
\(\widehat{\Gm^d}\) is locally noetherian and Jacobson, and its closed points are its
\(\ol\Ql\)-points.  The open subscheme \(U\) is therefore Jacobson, and its closed points are closed in
\(\widehat{\Gm^d}\), hence are characters outside \(\Delta(L)\) \cite[Tag~005X]{stacks-project}.

Let \(\chi\) be a closed point of \(U\), write \(A_\chi=\mathcal O_{U,\chi}\), and let
\((Q_m)_\chi\) denote the stalk of \(Q_m\) at \(\chi\), viewed as a complex of \(A_\chi\)-modules.
Write \(k(\chi)\) for the residue field of \(A_\chi\).  There is a canonical identification
\[
        Li_\chi^*Q_m\simeq (Q_m)_\chi\otimes^L_{A_\chi}k(\chi).
\]
The left-hand side vanishes by the two preceding paragraphs.
If \((Q_m)_\chi\) were nonzero, there would be a largest integer \(r\) with
\(H^r((Q_m)_\chi)\neq0\), because \((Q_m)_\chi\) is bounded with finitely generated cohomology
modules over the noetherian local ring \(A_\chi\).  In the hyper-Tor spectral sequence
\[
 E_2^{-a,b}=\Tor^{A_\chi}_a\bigl(H^b((Q_m)_\chi),k(\chi)\bigr)
 \Longrightarrow H^{b-a}\bigl(Li_\chi^*Q_m\bigr),
\]
the only possibly nonzero term in total degree \(r\) is \(H^r((Q_m)_\chi)\otimes_{A_\chi}k(\chi)\).
Nakayama's lemma then gives
\(H^r(Li_\chi^*Q_m)\simeq H^r((Q_m)_\chi)\otimes_{A_\chi}k(\chi)\neq0\), a contradiction.  Hence
\((Q_m)_\chi=0\) at every closed point \(\chi\) of \(U\).  The support of \(Q_m\) is a finite union
of supports of coherent cohomology sheaves, hence closed in \(U\).  A nonempty closed subset of the
Jacobson scheme \(U\) contains a closed point, so \(Q_m=0\).  Therefore
\[
        \Supp\Cone(\gamma_m)\subseteq\Delta(L).
\]

\emph{Localization.}  The morphism \(\gamma_m\) is an isomorphism outside \(\Delta(L)\), so \ref{item:GL5}
applied to \(\alpha_m\) shows that \(\alpha_m\) is an isomorphism in \(\Db(\Gm^d)_{\Delta(L)}\).
Equivalently, \(C_m=0\) there, that is, by \ref{item:GL4}, \(\Supp\mathscr M_!(C_m)\) and
\(\Supp\mathscr M_*(C_m)\) are contained in \(\Delta(L)\).  Since the localization functor is perverse
\(t\)-exact, the image of \({}^pH^j(C_m)\) in the localized heart is zero.  By \ref{item:GL4}, both Mellin
transforms of \({}^pH^j(C_m)\) are therefore supported in \(\Delta(L)\).

Fix \(\chi\notin\Delta(L)(\ol\Ql)\).  The preceding support statement and \ref{item:GL1} give
\[
        R\Gamma_c\bigl(\Gm^d,({}^pH^jC_m)_\chi\bigr)
        =R\Gamma\bigl(\Gm^d,({}^pH^jC_m)_\chi\bigr)=0
        \qquad(m\ge2,\ j\in\mathbb Z).
\]

By \eqref{eq:weak-locus}, every \(\chi\in X_w(L)\) lies outside \(\Delta(L)(\ol\Ql)\).  Since
\(L=M\oplus M^\vee\), one also has \(X_w(L)\subseteq X_w(M)\).  Fix \(\chi\in X_w(L)\).  The vanishing
above and the definitions of \S\ref{sec:negligible} give \(\chi\in\mathcal N({}^pH^jC_m)\) for all
\(j\) and hence \(\chi\in\mathcal N(C_m)\) for all \(m\ge2\).  Thus
\(\chi\in X_w(M)\cap\bigcap_{m\ge2}\mathcal N(C_m)\), and \ref{item:T5} shows that \(\chi\) is
unramified for \(M\).  Hence \(X_w(L)\subseteq X(M_0)\).
\end{proof}

\begin{proof}[Proof of Theorem~\ref{thm:unram-count}]
By \eqref{eq:weak-locus} and Lemma~\ref{lem:weakly-unramified-inclusion},
\(\widehat{\Gm^d}(\ol\Ql)\setminus X(M_0)\subseteq\Delta(L)(\ol\Ql)\), and
Theorem~\ref{thm:delta-count} applied to \(L\) gives
\[
        |(\widehat{\Gm^d}(\ol\Ql)\setminus X(M_0))(k_n)|
        \le|\Delta(L)(k_n)|
        \le\Phi^\Delta_{d}(\cu(L))\,q^{n(d-1)}
        \le\Phi^\Delta_{d}\bigl(F_{d}(\cu(M))\bigr)q^{n(d-1)}.
\]
We may therefore take
\[
        \Phi^{\mathrm{ur}}_{d}(C)
        =\Phi^\Delta_{d}\bigl(F_{d}(C)\bigr).\qedhere
\]
\end{proof}

\begin{lem}
\label{lem:identity-stalk}
For \(x\in\Gm^d\), let \(\delta_x\) denote the skyscraper sheaf at \(x\), and let \(1\in\Gm^d\) be the
identity.
Let \(P\) be a semisimple perverse sheaf on \(\Gm^d\), and suppose that \(\delta_1\) is not a direct
summand of \(P\).  Then
\[
        \cH^j(P)_1=0\qquad(j\ge0).
\]
\end{lem}

\begin{proof}
It suffices to consider a simple summand \(Q\) of \(P\).  If \(Q\) is punctual, then
\(Q\simeq\delta_x\) for some \(x\neq 1\), so its stalk at \(1\) vanishes.  Suppose that
\(\dim\Supp Q>0\).  The perverse support condition gives \(\cH^j(Q)=0\) for \(j>0\).  If
\(\cH^0(Q)_1\neq0\), the adjunction morphism for the inclusion \(i_1:\{1\}\hookrightarrow\Gm^d\) and
truncation give a nonzero morphism of perverse sheaves
\[
        Q\longrightarrow i_{1,*}i_1^*Q
        \longrightarrow i_{1,*}\tau^{\ge0}i_1^*Q
        =\delta_1\otimes_{\ol\Ql}\cH^0(Q)_1.
\]
Since \(Q\) is simple, this morphism is a monomorphism.  Hence \(\Supp Q\subseteq\{1\}\), a contradiction.
Thus \(\cH^0(Q)_1=0\).
\end{proof}

We use the following form of the Weyl criterion.  Let \(G\) be reductive and let \(K_G\) be a maximal
compact subgroup of \(G(\mathbb C)\).  Families of conjugacy classes
\(\theta_j(x)\in K_G^\sharp\), indexed by finite sets \(X_j\), are
\(\mu_G^\sharp\)-equidistributed if and only if
\[
        \frac{1}{|X_j|}\sum_{x\in X_j}\operatorname{Tr}\rho\bigl(\theta_j(x)\bigr)\longrightarrow0
\]
for every nontrivial irreducible algebraic representation \(\rho\) of \(G\).  

\begin{proof}[Proof of Corollary~\ref{cor:equidistribution}]
The argument follows the proof of \cite[Th.~4.19]{ForeyFresanKowalski2021}.
Put \(C=\sup_j\cu(M_j)\), and for every \(j\) write \(q_j=\#k_j\).
Theorem~\ref{thm:unram-count}, applied to \(M_{j,0}\), gives
\begin{equation}
\label{eq:cor-unram}
        |(\widehat{\Gm^d}(\ol\Ql)\setminus X(M_{j,0}))(k_j)|
        \le\Phi^{\mathrm{ur}}_{d}(C)\,q_j^{d-1},
\end{equation}
so \(|X_j|=(q_j-1)^d+O(q_j^{d-1})\), with constants depending only on \(d\) and \(C\).  Since
\(\#k_j\to\infty\), there is \(j_0\) such that \(X_j\neq\varnothing\) and \(|X_j|\asymp q_j^d\) for
every \(j\ge j_0\).  The finitely many \(j<j_0\) do not affect the limit and are omitted below.

Fix \(j\ge j_0\), and write \(k=k_j\), \(\bar k=\bar k_j\), \(q=q_j\), \(M_0=M_{j,0}\), \(M=M_j\), and
\(X=X_j\).  By \S\ref{sec:tannakian}, \(M\) is semisimple in \(\Perv_{\mathrm{int}}(\Gm^d)\).

Apply the Weyl criterion just recalled to the sets \(X_j\) and the unitary Frobenius conjugacy classes
\(\theta_j(\chi)=\Theta_{M_{j,0},k_j}(\chi)\) of \ref{item:T5}.  It suffices to show that, for every
nontrivial irreducible algebraic representation \(\rho\) of the common reductive group \(G\), the sums
in the criterion tend to zero.  We prove the stronger estimate
\begin{equation}
\label{eq:cor-weyl}
        \frac{1}{|X|}\sum_{\chi\in X}\operatorname{Tr}\rho\bigl(\Theta_{M_0,k}(\chi)\bigr)\ll_\rho q^{-1/2},
\end{equation}
with an implied constant independent of \(k\).
For the remainder of the proof, constants may also depend on \(G\) and on \(\rho\), but not on \(j\).
By \ref{item:T3}, the object \(N_0=\rho(M_0)\) is arithmetically semisimple and pure of weight zero.
Its geometric base change \(N\) is a direct summand of
\(L^{\ast_{\mathrm{int}}m(\rho)}\), where \(L=M\oplus M^\vee\) and \(m(\rho)\ge1\) depends only on
\(\rho\) and \(G\).  Put
\[
        C_\rho=
        F_{d,m(\rho),\mathrm{int}}\bigl(F_{d}(C)\bigr),
\]
with \(F_{d,m(\rho),\mathrm{int}}\) the convolution-power bound of
Corollary~\ref{cor:convolution-complexity} and \(F_{d}\) the function of
Lemma~\ref{lem:complexity-bounds}.  The direct-summand estimate \ref{item:C4} and
Corollary~\ref{cor:convolution-complexity} give \(\cu(N)\le C_\rho\).  This constant is independent of
\(j\) and of \(k_j\).  Write
\(\tau_i(\chi)=\operatorname{Tr}(\operatorname{Fr}_k\mid
\HH^i_c(\Gm^d,N_\chi))\), with \(\operatorname{Fr}_k\) the geometric Frobenius of \(k\).  By
\ref{item:T5}, one has \(\tau_i(\chi)=0\) for \(i\neq0\) and
\(\operatorname{Tr}\rho(\Theta_{M_0,k}(\chi))=\tau_0(\chi)\) for every \(\chi\in X\).

For \(0<|i|\le d\) put
\(\mathcal A_i=\{\chi\in\widehat{\Gm^d(k)}:\HH^i_c(\Gm^d,N_\chi)\neq0\}\).  By the definition of
complexity \ref{item:C1}, the tensor-product estimate
\ref{item:C3}, and Lemma~\ref{lem:kummer-complexity}, there is a constant \(B_\rho\), independent
of \(j\) and of \(k_j\), such that
\[
        \dim\HH^i_c(\Gm^d,N_\chi)\le\cu(N_\chi)\le B_\rho
        \qquad(i\in\mathbb Z,\ \chi\in\widehat{\Gm^d}(\ol\Ql)).
\]
We use the following three estimates.
\begin{enumerate}
\item[(a)] \(\mathcal A_i\subseteq\mathcal V_{d,|i|}(N)\), so Theorem~\ref{thm:stratified} gives
\(|\mathcal A_i|\le\Psi_{d}(C_\rho)\,q^{d-|i|}\).
\item[(b)] As \(N_0\) is pure of weight zero, \(\HH^i_c(\Gm^d,N_\chi)\) has weights \(\le i\)
\cite[3.3.1]{Del80}, \cite[5.1.13]{BBDG18}, so
\(|\tau_i(\chi)|\le B_\rho\,q^{i/2}\).
\item[(c)] \(|\widehat{\Gm^d(k)}\setminus X|\le\Phi^{\mathrm{ur}}_{d}(C)q^{d-1}\) by
\eqref{eq:cor-unram}.
\end{enumerate}

For \(i\neq0\), the function \(\tau_i\) is supported on
\(\mathcal A_i\subseteq\widehat{\Gm^d(k)}\setminus X\).  Hence
\[
\begin{aligned}
        \frac{1}{|X|}\sum_{\chi\in X}\tau_0(\chi)
        =\ &\frac{1}{|X|}\sum_{\chi\in\widehat{\Gm^d(k)}}\sum_{|i|\le d}(-1)^i\tau_i(\chi)
        -\frac{1}{|X|}\sum_{0<|i|\le d}(-1)^i\sum_{\chi\in\mathcal A_i}\tau_i(\chi)\\
        &-\frac{1}{|X|}\sum_{\chi\in\widehat{\Gm^d(k)}\setminus X}\tau_0(\chi).
\end{aligned}
\]
By the Grothendieck--Lefschetz trace formula \(\sum_{|i|\le d}(-1)^i\tau_i(\chi)=\sum_{x\in \Gm^d(k)}
t_{N_0}(x;k)\,\chi(x)\), where \(t_{N_0}\) denotes the trace function of the arithmetic object
\(N_0\).  Summing over
\(\chi\in\widehat{\Gm^d(k)}\) and using
\(\sum_\chi\chi(x)=|\widehat{\Gm^d(k)}|\) for \(x=1\) and \(\sum_\chi\chi(x)=0\) otherwise, the first term
equals \((|\widehat{\Gm^d(k)}|/|X|)\,t_{N_0}(1;k)\).

It remains to bound the identity contribution \(t_{N_0}(1;k)\).  By \ref{item:T1}, the tannakian unit is
the skyscraper sheaf \(\delta_1\) at the identity.  Since
\(G^{\mathrm{ari}}_M=G^{\mathrm{geo}}_M=G\) and \(\rho\) is nontrivial and irreducible,
\(\operatorname{Hom}(\delta_1,N)=0\) in \(\langle M\rangle\).  Both inclusions
\(\langle M\rangle\subseteq\Perv_{\mathrm{int}}(\Gm^d)\subseteq\Perv(\Gm^d)\) are full, so \(N\) has no
direct summand isomorphic to \(\delta_1\).  By \ref{item:T3} and \S\ref{sec:tannakian}, \(N\) is semisimple.
Lemma~\ref{lem:identity-stalk} therefore gives \(\cH^j(N)_1=0\) for \(j\ge0\).  Thus only
degrees \(j\le-1\) contribute to \(t_{N_0}(1;k)\).  Since \(N_0\) is pure of weight zero, the Frobenius
eigenvalues on \(\cH^j(N)_1\) have weights at most \(j\) \cite[1.2]{Del80}, \cite[5.1.8]{BBDG18} and
therefore have absolute value at most \(q^{-1/2}\) in these degrees.  Consequently,
\[
        |t_{N_0}(1;k)|
        \le q^{-1/2}\sum_{j\le-1}\dim\cH^j(N)_1
        \le F_{d}(C_\rho)q^{-1/2},
\]
because
\[
 \sum_{j\le-1}\dim\cH^j(N)_1
 \le F_{d}(\cu(N))
 \le F_{d}(C_\rho),
\]
where the first inequality is Lemma~\ref{lem:complexity-bounds}(i) and the second is monotonicity.  As
\(|\widehat{\Gm^d(k)}|\asymp|X|\), the first term is \(\ll_\rho q^{-1/2}\).  By (a) and (b) the
second term is
\[
        \ll_\rho\frac{1}{q^d}\sum_{1\le i\le d}q^{d-i}\,q^{i/2}\ll_\rho q^{-1/2},
\]
the terms with \(i<0\) vanishing by Artin affine vanishing.  By (b) and (c) the third
term is \(\ll_\rho q^{d-1}\cdot B_\rho/q^d\ll_\rho q^{-1}\).  This proves \eqref{eq:cor-weyl}.

The trivial representation gives average one, so the Weyl criterion yields the equidistribution of the
classes \(\Theta_{M_0,k}(\chi)\), \(\chi\in X(M_0)(k)\), in \(K_G^\sharp\) for the measure
\(\mu_G^\sharp\) as \(\#k_j\to\infty\).
\end{proof}

\appendix

\section{Compatibilities of the forget-supports transformation}
\label{app:forget-supports}

The transformation \(\varphi_f\) is defined in \eqref{eq:forget-supports}.
All morphisms in this appendix are separated and of finite type over \(\bar k\).  For
\(f:X\to S\), \(g:Y\to T\), \(K\in\Db(X)\), and \(K'\in\Db(Y)\), let
\[
 \kappa_!:Rf_!K\boxtimes Rg_!K'\xrightarrow{\sim}R(f\times g)_!(K\boxtimes K')
\]
be the K\"unneth isomorphism \cite[Exp.~XVII, Th.~5.4.3]{SGA4}.  Let
\[
 \kappa_*:Rf_*K\boxtimes Rg_*K'\longrightarrow R(f\times g)_*(K\boxtimes K')
\]
be the morphism adjoint to the product of the counits.

\begin{lem}
\label{lem:forget-supports}
\begin{enumerate}
\item The transformation \(\varphi_f\) does not depend on the compactification.  Under the canonical
identifications, \(\varphi_f=\operatorname{id}\) for \(f\) proper, and \(\varphi_f=\varphi_j\) for \(f=j\)
an open immersion.
\item For \(p\colon X\to Y\) and \(r\colon Y\to Z\), under the canonical isomorphisms
\(R(rp)_!\simeq Rr_!\,Rp_!\) and \(R(rp)_*\simeq Rr_*\,Rp_*\),
\[
        \varphi_{rp}(K)=\varphi_r\bigl(Rp_*K\bigr)\circ Rr_!\bigl(\varphi_p(K)\bigr).
\]
\end{enumerate}
\end{lem}

\begin{proof}
For an open immersion \(j\), the morphism \(\varphi_j(K)\) is characterized by
\(j^*\varphi_j(K)=\operatorname{id}_K\).  We use this characterization throughout.

(1)  Let \((j_i,\bar f_i)\), \(i=1,2\), be two compactifications of \(f\), and let \(\bar X_3\) be the
closure of the image of \((j_1,j_2)\colon X\to\bar X_1\times_Y\bar X_2\).  Write
\(j_3\colon X\to\bar X_3\) for the induced map and \(r_i\colon\bar X_3\to\bar X_i\) for the proper
projections.  Since \(\bar X_{3-i}\to Y\) is separated, the graph of
\(j_{3-i}\circ j_i^{-1}\) is closed over \(j_i(X)\), so
\(r_i^{-1}(j_i(X))=j_3(X)\).  Thus \(j_3\) is an open immersion with dense image, and \(\bar X_3\) is a
third compactification.

It remains to compare \((j',\bar f\circ q)\) with \((j,\bar f)\), where \(q\) is proper,
\(q\circ j'=j\), and \(q^{-1}(j(X))=j'(X)\).  Proper base change gives a canonical isomorphism
\[
        j_!K\xrightarrow{\ \sim\ }Rq_*j'_!K.
\]
Under this isomorphism and \(Rq_*Rj'_*=Rj_*\), the morphism \(Rq_*(\varphi_{j'}(K))\) pulls back along
\(j\) to \(\operatorname{id}_K\), and therefore equals \(\varphi_j(K)\).  Applying \(R\bar f_*\) proves
independence.  If \(f\) is proper, take \(j=\operatorname{id}\).  If \(f=j\) is an open immersion, the
transformation just constructed pulls back along \(j\) to the identity and hence equals \(\varphi_j\) by
the characterization at the start of the proof.

(2)  Choose compactifications \(p=\bar p\circ j\) and \(r=\bar r\circ j'\), and choose a
compactification
\[
        \bar X\xrightarrow{\ j''\ }\bar X''\xrightarrow{\ q\ }\bar Y'
\]
of \(j'\circ\bar p\colon\bar X\to\bar Y'\), with \(j''\) an open immersion with dense image and \(q\) proper
\cite[Tag~0ATU]{stacks-project}.  Since \(\bar X\) is proper over
\(Y\) and \(q^{-1}(j'(Y))\) is separated over \(Y\), the induced morphism
\(\bar X\to q^{-1}(j'(Y))\) is proper.  Its image \(j''(\bar X)\) is therefore closed.  It is also dense
because \(j''(\bar X)\) is dense in \(\bar X''\).  Thus \(q^{-1}(j'(Y))=j''(\bar X)\).  Then
\(rp=(\bar r\circ q)\circ(j''\circ j)\) is a compactification of \(rp\).  The characterization above gives
\[
        \varphi_{j''j}(K)=\varphi_{j''}(Rj_*K)\circ j''_!(\varphi_j(K)).
\]
For \(M\in\Db(\bar X)\), proper base change and the equality
\(q^{-1}(j'(Y))=j''(\bar X)\) give a natural isomorphism
\[
        j'_!R\bar p_*M\xrightarrow{\ \sim\ }Rq_*j''_!M.
\]
Under this isomorphism and \(Rq_*Rj''_*=Rj'_*R\bar p_*\), the characterization of \(\varphi_{j'}\) gives
\[
        Rq_*(\varphi_{j''}(M))=\varphi_{j'}(R\bar p_*M).
\]
Apply \(R(\bar r\circ q)_*\) to the first identity and use the naturality of the intervening isomorphism.  The
composition isomorphisms of \cite[Exp.~XVII, \S5.1]{SGA4} then give
\[
        \varphi_{rp}(K)
        =R\bar r_*\bigl[Rq_*\varphi_{j''}(Rj_*K)\bigr]
         \circ R\bar r_*\bigl[Rq_*j''_!(\varphi_j(K))\bigr]
        =\varphi_r(Rp_*K)\circ Rr_!(\varphi_p(K)).
\]
\end{proof}

\begin{lem}
\label{lem:forget-supports-products}
\begin{enumerate}
\item For \(f:X\to S\), \(g:Y\to T\), \(K\in\Db(X)\), and \(K'\in\Db(Y)\), one has
\[
        \varphi_{f\times g}(K\boxtimes K')\circ\kappa_!
        =\kappa_*\circ\bigl(\varphi_f(K)\boxtimes\varphi_g(K')\bigr).
\]
\item Let \(p:X\to Y\), let \(\cL\) be a lisse \(\ol\Ql\)-sheaf on \(Y\), and let
\(K\in\Db(X,\ol\Ql)\).  The natural morphism
\[
 \pi_*\colon Rp_*K\otimes\cL\xrightarrow{\ \sim\ }Rp_*(K\otimes p^*\cL)
\]
is an isomorphism \cite[Prop.~A.1.5(i)]{GabberLoeser1996}.  Under this isomorphism and the
projection-formula identification for \(Rp_!\), the morphism \(\varphi_p(K\otimes p^*\cL)\) identifies
with \(\varphi_p(K)\otimes\operatorname{id}_{\cL}\).
\end{enumerate}
\end{lem}

\begin{proof}
(1)  Compactify \(f=\bar f\circ j_X\) and \(g=\bar g\circ j_Y\).  Then
\(f\times g=(\bar f\times\bar g)\circ j\) with \(j=j_X\times j_Y\) an open immersion and
\(\bar f\times\bar g\) proper.  Under the canonical identification
\(j_!(K\boxtimes K')\simeq j_{X!}K\boxtimes j_{Y!}K'\), let
\(c\colon Rj_{X*}K\boxtimes Rj_{Y*}K'\to Rj_*(K\boxtimes K')\) be the morphism adjoint to the identity
after pullback by \(j\).  The characterization of the forget-supports transformation for an open
immersion gives
\[
 \varphi_j(K\boxtimes K')
 =c\circ\bigl(\varphi_{j_X}K\boxtimes\varphi_{j_Y}K'\bigr).
\]
Apply \(R(\bar f\times\bar g)_*\) and the K\"unneth isomorphism for the proper morphisms
\(\bar f,\bar g\) \cite[Exp.~XVII, Th.~5.4.3]{SGA4}.  The left side becomes
\(\varphi_{f\times g}(K\boxtimes K')\circ\kappa_!\), while \(c\) induces \(\kappa_*\).  This gives~(1).

(2)  We first prove the formula for \(Rp_*\).  Since \(\cL\) is lisse of finite rank, it is dualizable.
Write \(\cL^\vee\) for its dual.  For every \(M\in\Db(Y,\ol\Ql)\), duality for \(\cL\) and
\(p^*\cL\), together with the adjunction \(p^*\dashv Rp_*\), gives natural isomorphisms
\[
\begin{aligned}
 \operatorname{Hom}_Y(M,Rp_*K\otimes\cL)
 &\simeq \operatorname{Hom}_Y(M\otimes\cL^\vee,Rp_*K)\\
 &\simeq \operatorname{Hom}_X(p^*M\otimes p^*\cL^\vee,K)\\
 &\simeq \operatorname{Hom}_X(p^*M,K\otimes p^*\cL)\\
 &\simeq \operatorname{Hom}_Y(M,Rp_*(K\otimes p^*\cL)).
\end{aligned}
\]
By Yoneda, these give an isomorphism
\[
        Rp_*K\otimes\cL\xrightarrow{\ \sim\ }Rp_*(K\otimes p^*\cL).
\]
Its adjoint is the tensor product of the counit \(p^*Rp_*K\to K\) with
\(\operatorname{id}_{p^*\cL}\), so this is the usual natural morphism.

Choose a compactification \(p=\bar p\circ j\) in the definition of \(\varphi_p\), so that
\(\varphi_p=R\bar p_*(\varphi_j)\), and put \(\cM=\bar p^*\cL\).  The \(!\)-projection formula and the
formula for \(Rj_*\) just proved give
\[
 j_!K\otimes\cM\xrightarrow{\ \sim\ }j_!(K\otimes j^*\cM),
 \qquad
 Rj_*K\otimes\cM\xrightarrow{\ \sim\ }Rj_*(K\otimes j^*\cM).
\]
Under these isomorphisms, both \(\varphi_j(K\otimes j^*\cM)\) and
\(\varphi_j(K)\otimes\operatorname{id}_{\cM}\) pull back to the identity, and hence they agree.  Apply
\(R\bar p_*\) and the proper projection formula \cite[Exp.~XVII, Prop.~5.2.9]{SGA4},
\cite[Prop.~A.1.5(ii)]{GabberLoeser1996} to obtain~(2).
\end{proof}

The next elementary lemma will be applied to \(R\bar j_{x,*}\) in the proof of
Lemma~\ref{lem:exceptional}.

\begin{lem}
\label{lem:constant-tensor}
Let \(f\colon X\to Y\) be any morphism of schemes and let \(V\) be a finite-dimensional
\(\ol\Ql\)-vector space.  For \(Z=X,Y\), let \(V_Z\) be the constant sheaf on
\(Z_{\mathrm{pro\acute et}}\) associated with \(V\).  The tensor products below are derived tensor
products of \(\ol\Ql\)-sheaves.  For \(K\in D(X_{\mathrm{pro\acute et}},\ol\Ql)\), the natural morphism
\begin{equation}
\label{eq:constant-tensor}
        \theta_{K,V}\colon
        Rf_*K\otimes_{\ol\Ql}V_Y
        \longrightarrow
        Rf_*(K\otimes_{\ol\Ql}V_X)
\end{equation}
is an isomorphism, naturally in \(K\) and in \(V\).
\end{lem}

\begin{proof}
The morphism is adjoint to
\[
 f^*(Rf_*K\otimes_{\ol\Ql}V_Y)
 \simeq f^*Rf_*K\otimes_{\ol\Ql}V_X
 \longrightarrow K\otimes_{\ol\Ql}V_X,
\]
where the last arrow is induced by the counit \(f^*Rf_*K\to K\).  After choosing a basis of \(V\),
\eqref{eq:constant-tensor} becomes the canonical morphism
\[
        (Rf_*K)^{\oplus\dim V}\longrightarrow Rf_*(K^{\oplus\dim V}),
\]
which is an isomorphism because \(Rf_*\) is additive.  Naturality follows from the adjoint
construction.
\end{proof}

\section{Continuous cohomology of inertia groups}
\label{app:boundary}

This appendix computes continuous cohomology for profinite groups whose finite quotients have order
prime to \(\ell\) and for \(\mathbb Z_\ell\).  The results are used in the proof of
Lemma~\ref{lem:inertia-eigenvalue}.

\begin{lem}
\label{lem:prime-to-ell-vanishing}
Let \(H\) be a profinite group all of whose finite quotients have order prime to \(\ell\), and let \(M\)
be a finite-dimensional continuous \(\ol\Ql\)-representation of \(H\).  Then continuous cohomology is
concentrated in degree \(0\).
\[
        R\Gamma_{\mathrm{cts}}(H,M)\simeq M^H,
        \qquad
        H^a_{\mathrm{cts}}(H,M)=0\quad(a>0).
\]
The same conclusion holds when \(M\) is a finite discrete \(H\)-module of \(\ell\)-power order.
\end{lem}

\begin{proof}
Choose a finite extension \(E/\mathbb Q_\ell\) over which \(M\) is defined
(\S\ref{sec:sheaves}), a uniformizer \(\varpi\) of \(E\), and an \(H\)-stable
\(\mathcal O_E\)-lattice \(\Lambda\subset M\).  Continuous cochains with
values in \(\Lambda/\varpi^n\Lambda\) factor through finite quotients \(Q\) of \(H\).  The order of every
such \(Q\) is prime to \(\ell\), hence is invertible on \(\Lambda/\varpi^n\Lambda\), and the norm
idempotent
\[
        e_Q=\frac{1}{\lvert Q\rvert}\sum_{g\in Q}g
\]
projects onto \(Q\)-invariants.  Thus \(Q\)-invariants are exact on these modules.  Passing over the
finite quotients gives
\[
        H^a_{\mathrm{cts}}(H,\Lambda/\varpi^n\Lambda)=0
        \qquad(a>0).
\]
Let
\[
 K_n=\ker\bigl(\Lambda/\varpi^{n+1}\Lambda\longrightarrow\Lambda/\varpi^n\Lambda\bigr).
\]
The same argument gives \(H^1_{\mathrm{cts}}(H,K_n)=0\).  The long exact sequence therefore shows
that the transition maps on
\(H^0_{\mathrm{cts}}(H,\Lambda/\varpi^n\Lambda)=(\Lambda/\varpi^n\Lambda)^H\) are surjective.  In
Jannsen's exact sequence
\cite[(2.1)]{Jannsen88}
\[
        0\to{\varprojlim_n}^{1}\,H^{a-1}_{\mathrm{cts}}(H,\Lambda/\varpi^n\Lambda)
        \to H^a_{\mathrm{cts}}(H,\Lambda)
        \to\varprojlim_n H^a_{\mathrm{cts}}(H,\Lambda/\varpi^n\Lambda)\to0,
\]
the right term vanishes for \(a>0\).  If \(a>1\), the left term vanishes by the same cohomology
vanishing.  If \(a=1\), it is
\({\varprojlim_n}^{1}(\Lambda/\varpi^n\Lambda)^H\), which vanishes because the transition maps are
surjective.  Extending scalars to \(\ol\Ql\) proves the assertion.  For a finite discrete module \(M\)
of \(\ell\)-power order, continuous cochains factor through finite quotients of \(H\), and the norm
idempotent gives the same conclusion.
\end{proof}

\begin{lem}
\label{lem:zell-cohomology}
Let \(\Gamma\simeq\mathbb Z_\ell\) with topological generator \(\gamma\), and let \(M\) be a
finite-dimensional continuous \(\ol\Ql\)-representation of \(\Gamma\).  Then
\[
        R\Gamma_{\mathrm{cts}}(\Gamma,M)
        \simeq
        \bigl[M\xrightarrow{\gamma-1}M\bigr],
\]
where the first copy of \(M\) in the two-term complex on the right is in degree \(0\).
The same description holds when \(M\) is instead a finite discrete \(\Gamma\)-module of \(\ell\)-power
order.
\end{lem}

\begin{proof}
Choose an \(E\)-model \(M_E\) of \(M\), for a finite extension \(E/\mathbb Q_\ell\) with
uniformizer \(\varpi\), and a \(\Gamma\)-stable \(\mathcal O_E\)-lattice \(\Lambda\subset M_E\)
(\S\ref{sec:sheaves}).  Put \(\Lambda_n=\Lambda/\varpi^n\Lambda\).  For every \(n\), the completed
group algebra satisfies
\[
        (\mathcal O_E/\varpi^n)[[\Gamma]]\simeq(\mathcal O_E/\varpi^n)[[T]],
        \qquad \gamma\longmapsto1+T,
\]
and the augmentation identifies the trivial module \(\mathcal O_E/\varpi^n\) with the quotient by
\(T=\gamma-1\).  Multiplication by \(T\) is injective in a formal power-series ring, so
\[
 0\longrightarrow(\mathcal O_E/\varpi^n)[[\Gamma]]
 \xrightarrow{\ \gamma-1\ }(\mathcal O_E/\varpi^n)[[\Gamma]]
 \xrightarrow{\ \mathrm{aug}\ }(\mathcal O_E/\varpi^n)\longrightarrow0
\]
is a free resolution of the trivial module.  Applying continuous module homomorphisms into
\(\Lambda_n\) gives, functorially and compatibly in \(n\),
\[
        R\Gamma_{\mathrm{cts}}(\Gamma,\Lambda_n)
        \simeq
        \bigl[\Lambda_n\xrightarrow{\gamma-1}\Lambda_n\bigr].
\]
For the identification of the cohomology of this two-term complex with continuous cohomology at each
finite level see \cite[Prop.~1.7.7]{NSW2008}.
The transition maps of these two-term complexes are termwise surjective, so the derived inverse
limit is the termwise inverse limit, and the definition in \S\ref{sec:sheaves} gives
\[
        R\Gamma_{\mathrm{cts}}(\Gamma,M_E)
        \simeq
        \bigl[M_E\xrightarrow{\gamma-1}M_E\bigr].
\]
Extending scalars to \(\ol\Ql\) proves the assertion.  For a finite discrete \(\Gamma\)-module \(A\)
of \(\ell\)-power order, choose \(n\) with
\(\ell^nA=0\).  The same free resolution over \((\mathbb Z/\ell^n)[[\Gamma]]\) applies directly.
\end{proof}

\bibliographystyle{alpha-custom}
\bibliography{bib/references}

\end{document}